\documentclass[11pt,american,preprint]{elsarticle}
\makeatletter
\def\ps@pprintTitle{%
 \let\@oddhead\@empty
 \let\@evenhead\@empty
 \def\@oddfoot{\centerline{\thepage}}%
 \let\@evenfoot\@oddfoot}
\makeatother

\usepackage[margin=1in]{geometry}

\usepackage{xcolor}
\usepackage{lipsum}
\usepackage{changepage}
\usepackage{centernot}
\usepackage{apptools}
\usepackage{mathrsfs}

\usepackage[T1]{fontenc}
\usepackage[latin9]{luainputenc}
\usepackage{amsmath}
\usepackage{float}
\usepackage{graphicx}
\usepackage{amsthm}
\usepackage{amssymb}
\usepackage{stmaryrd}
\usepackage{esint}
\usepackage{nameref}
\usepackage{hyperref} 
\usepackage{cleveref} 

\usepackage[shortlabels]{enumitem}
\usepackage{chngcntr}

\newsavebox{\foobox}

\usepackage{amssymb}

\usepackage{array}
\newcolumntype{M}[1]{>{\centering\arraybackslash}m{#1}}

\usepackage{bbm} 
\usepackage{babel}
\usepackage{fancyhdr}
\usepackage{chngcntr}
\usepackage{apptools}
\makeatletter
\numberwithin{equation}{section}
\theoremstyle{plain}
\newtheorem{thm}{\protect\theoremname}[section]

\theoremstyle{plain*}
\newtheorem*{thm*}{\protect\theoremname}

\theoremstyle{plain}
\newtheorem{lem}[thm]{\protect\lemmaname}
  
\theoremstyle{plain*}
\newtheorem*{lem*}{\protect\lemmaname}  
  
  \theoremstyle{plain}
  \newtheorem{prop}[thm]{\protect\propositionname}
  
    \theoremstyle{plain*}
  \newtheorem*{prop*}{\protect\propositionname}

\theoremstyle{remark}
\newtheorem{question}[thm]{Question}
\theoremstyle{remark*}
\newtheorem*{question*}{Question} 
\theoremstyle{remark}
\newtheorem{rem}[thm]{\protect\remarkname}
\theoremstyle{remark*}
\newtheorem*{rem*}{\protect\remarkname}
\theoremstyle{remark}
\newtheorem{example}[thm]{Example}
\theoremstyle{remark*}
\newtheorem*{example*}{\protect\examplename}

\theoremstyle{plain}
\newtheorem{cor}[thm]{\protect\corollaryname}
\providecommand{\corollaryname}{Corollary}

\theoremstyle{plain}
\newtheorem{conjecture}[thm]{\protect\conjecturename}
\providecommand{\conjecturename}{Conjecture}
  
\theoremstyle{definition}
  
\makeatother

\theoremstyle{plain} 
\newcommand{\thistheoremname}{}
\newtheorem{genericthm}[thm]{\thistheoremname}

\newtheorem*{genericthm*}{\thistheoremname}
\newenvironment{namedthm*}[1]
  {\renewcommand{\thistheoremname}{#1}%
   \begin{genericthm*}}
  {\end{genericthm*}}

\makeatother

\providecommand{\lemmaname}{Lemma}
\providecommand{\propositionname}{Proposition}
\providecommand{\remarkname}{Remark}
\providecommand{\theoremname}{Theorem}
\providecommand{\corollaryname}{Corollary}

\newcommand{\R}{\mathbb{R}}
\newcommand{\N}{\mathbb{N}}

\newcommand{\Z}{\mathbb Z}

\newcommand{\FF}{{\mathbb F_2^\omega}}

\newcommand\precdot{\mathrel{\ooalign{$\prec$\cr
  \hidewidth\raise0ex\hbox{$\cdot\mkern0.5mu$}\cr}}}
\newcommand\preceqdot{\mathrel{\ooalign{$\preceq$\cr
  \hidewidth\raise0.225ex\hbox{$\cdot\mkern0.5mu$}\cr}}}

\title{Polynomial maps which are not good for nice recurrence and applications}
\author{R. Zelada}
\AtAppendix{\counterwithin{thm}{section}}

\begin{document}
\begin{abstract}
Let $\mathbb F_2$ be the finite field with two elements and let $\FF$ denote the countably infinite-dimensional vector space over $\mathbb F_2$. We show that, unlike the case of polynomial maps $p:\Z\rightarrow\Z$ vanishing at zero which are always good for nice recurrence, 
there are polynomials $p:\FF\rightarrow \FF$ with $p(0_\FF)=0_\FF$ which fail to have this property. This  disproves a conjecture of Bergelson and McCutcheon (c. 2000), which predicted that for any countably infinite abelian groups $H$ and $G$, every polynomial map $p:H\to G$ with $p(0_H)=0_G$ is good for nice recurrence. Moreover, we develop a dynamical mechanism which shows that the magnitude of intersections along polynomial paths  is  degree-sensitive (even when one considers only weakly mixing systems).\\
Among other things, we also show that the Furstenberg-S{\'a}rk{\"o}zy theorem for  $\FF$-valued polynomials of degree at most $d$ vanishing at zero is equivalent to a $d$-dimensional symmetric-difference weakening of the density polynomial Hales-Jewett conjecture. Thus, as we explain in detail in this paper,
our observations not only shed new light on the phenomenon of polynomial recurrence 
but also 
constrain possible strategies for  proving or disproving   the density polynomial Hales-Jewett conjecture. 
\end{abstract}
\maketitle
\tableofcontents
\section{Introduction}
The Furstenberg-S{\'a}rk{\"o}zy theorem  \cite[Theorem 3.16]{Fbook}, \cite{sarkozy1978difference}, which can be viewed as a polynomial extension of the classical Poincar{\'e} recurrence theorem, can be stated as follows.
\begin{thm}[Furstenberg-S{\'a}rk{\"o}zy Theorem] \label{0.FSTheorem}
    Let $p\in\Z[x]$ be a non-constant polynomial with zero constant term and let $(X,\mathcal A,\mu,T)$ be an invertible probability preserving system. For any $A\in\mathcal A$ with $\mu(A)>0$, there are infinitely many $n\in\Z$ with $\mu(A\cap T^{-p(n)}A)>0$.
\end{thm}
A careful analysis of Furstenberg's proof of  \cref{0.FSTheorem} yields the following stronger result dealing with what we  call "functions  good for nice recurrence". Given a countably infinite set $S$ and an abelian group $G$, a map $p:S\rightarrow G$ is said to be \textit{good for nice recurrence}  if for any invertible probability preserving system $(X,\mathcal A,\mu,(T^g)_{g\in G})$, any $A\in\mathcal A$ with $\mu(A)>0$, and any $\epsilon>0$ there are infinitely many  $s\in S$ for which 
$$
\mu(A\cap T^{-p(s)}A)>\mu^2(A)-\epsilon.
$$
If one only has that $\mu(A\cap T^{-p(s)}A)>0$ for infinitely many $s\in S$ instead, we will say that $p$ is \textit{good for recurrence}.
\begin{cor}\label[corollary]{0.cor:FScor}
Let $p\in\Z[x]$ be a non-constant polynomial with zero constant term. The polynomial $p:\Z\rightarrow \Z$ is good for nice recurrence. 
\end{cor}
More generally, it is known that a  polynomial map $p:\Z\rightarrow \Z$ is good for nice recurrence if and only if it is intersective (i.e. for each $M\in\N$, there is an $n\in\Z$ with $p(n)\equiv 0\mod M$). Actually, one also has that $p$ is good for recurrence if and only if it is intersective \cite{KamaeFranceNiceRec1978} (see also Theorem 4.5 and Remark 4.6 in \cite{AckBer2025RingsOfIntegers}). Thus, in  a certain sense, the \textit{algebraic properties} of $p$ completely determine its recurrence properties. A conjecture of V. Bergelson and R. McCutcheon \cite[Conjecture 2]{berMcCuIPPolySzemeredi} predicted that this pattern generalizes to arbitrary group-theoretic polynomial maps vanishing at zero or, equivalently, that Corollary \ref{0.cor:FScor} can be extended to arbitrary countable abelian groups: for any given countably infinite abelian groups $H$ and $G$, every polynomial map $p:H\rightarrow G$ with $p(0_H)=0_G$ is good for nice recurrence.\\
Several results supported this prediction by verifying it for natural families of polynomial expressions on various abelian groups  (see, for example, \cite{BFM}, \cite{BHM}, \cite{BDonaldRobertsonIP_r}, \cite{McCutWSarkozy2014}, \cite{AckBer2025RingsOfIntegers}). The Polynomial Hales-Jewett theorem \cite{BerLeibPolyHJ} provided further evidence for the plausibility of  \cite[Conjecture 2]{berMcCuIPPolySzemeredi}. However, an example in  \cite{BFM} already foreshadowed  the difficulties  that can arise for infinitely generated abelian groups.
In fact, our main result in this paper implies that \Cref{0.cor:FScor} cannot be extended to arbitrary countably infinite abelian groups (and, so, \cite[Conjecture 2]{berMcCuIPPolySzemeredi} is false).\\
Our main result,  Theorem \hyperlink{0.thm:FailureOfPolynomialKhintchine}{A}, provides a family of  counterexamples to  \cite[Conjecture 2]{berMcCuIPPolySzemeredi} each of which is associated with a group-theoretic polynomial $p_d:\FF\rightarrow \FF$ of degree $d$, $d\geq 2$ (here and throughout this work, $\FF$ denotes the vector space with countably infinite dimension over the field with exactly two elements $\mathbb F_2$). The subsequent analysis of these examples reveals several interesting features of $\FF$-valued polynomial recurrence, including connections with the density polynomial Hales-Jewett conjecture and with a wider class of polynomial-like expressions known as VIP-systems. \\
More precisely, these counterexamples demonstrate that, over $\FF$, polynomial recurrence is not merely a qualitative issue: its strength depends on the degree and appears to be governed by \textit{combinatorial/algebraic} obstructions. Moreover, unlike the situation for polynomial maps $p:\Z\rightarrow\Z$, these examples indicate that an $\FF$-valued polynomial could be good for recurrence while failing to be good for nice recurrence.
Furthermore, as we show in Appendix \ref{C.Sec}, the recurrence properties of the  polynomials $p_d$, $d\geq 2$, have a  \textit{genuinely} combinatorial/algebraic character:
$p_d$ is good for recurrence if and only if every $\FF$-valued VIP-system of degree at most $d$ is good for recurrence, and this is equivalent to 
the following symmetric-difference weakening  of the density polynomial Hales-Jewett conjecture (DPHJ).
\begin{adjustwidth}{2em}{2em}
{\bf A symmetric-difference weakening  of DPHJ of dimension $d$:} For any $\delta>0$ there is an $N\in\N$ such that if $\mathcal S\subseteq \mathcal P(\{1,...,N\}^d)$ satisfies $|\mathcal S|\geq \delta 2^{N^d}$, then one can find $E,F\in\mathcal S$ and a non-empty $\gamma\subseteq \{1,...,N\}$ with $E\triangle F=\gamma^d$ (Cf. \cite[Gil Kalai comment, Mar/05/2013]{GowersDPHJjBlog}).\\
\end{adjustwidth}
The preceding paragraphs give a brief outline of the main features of our examples. In the next six subsections, we make this outline precise: we first state Theorem \hyperlink{0.thm:FailureOfPolynomialKhintchine}{A} and discuss its consequences for $\FF$-valued polynomial recurrence; then we relate these results to DPHJ, VIP-systems, and IP-/IP$^*$-convergence.

\subsection{ Background on polynomial maps and a consequence of the conjecture of Bergelson-McCutcheon}
In this subsection we review the necessary background on  group-theoretic polynomials and   record a consequence of \cite[Conjecture 2] {berMcCuIPPolySzemeredi} which is disproved by Theorem \hyperlink{0.thm:FailureOfPolynomialKhintchine}{A}. As we will see, this consequence would have provided  a far-reaching  extension of the Furstenberg-S{\'a}rk{\"o}zy theorem to arbitrary countable abelian groups, and thus stating it here helps explain the interest in \cite[Conjecture 2] {berMcCuIPPolySzemeredi}.\\

Let $G$ and $H$ be  abelian groups and let $p:H\rightarrow G$ be a function. For any $h\in H$, we define $D_hp:H\rightarrow G$, the \textit{discrete derivative of $p$ at $h$},  by the rule 
$$
D_hp(x)=p(x+h)-p(x).
$$
We say that $p$ is a polynomial mapping if there exists a $t\in \N$  with the property that for any $h_0,...,h_t\in H$ and any $g\in H$, 
\begin{equation}\label{0.eq:DefnDP}
D_{h_t}\cdots D_{h_0}p(g)=0_{G}.
\end{equation}
The \textit{degree of $p$}, denoted $\deg(p)$, is the  least $t\in\N$ for which \eqref{0.eq:DefnDP} holds.
When $H=G=\Z$, every polynomial mapping $p:\Z\rightarrow \Z$ belongs to the group generated by the "algebraic polynomials"
$$
\binom{x}{d}=\frac{(x-d+1)\cdots(x-1)(x)}{d!},\,d\in\N\cup\{0\},
$$
where we adopt the conventions that $\binom{x}{1}=x$ and $\binom{x}{0}=1$ \cite[Proposition 2.6]{leibman1998polynomial}.\\

The next proposition is the consequence of \cite[Conjecture 2]{berMcCuIPPolySzemeredi} mentioned above.  It mirrors  Corollary  \ref{0.cor:FScor} by  asserting that if \cite[Conjecture 2]{berMcCuIPPolySzemeredi} holds (and $H$ is infinite),  then any polynomial $p:H\rightarrow G$ with $p(O_H)=O_G$ is good for nice recurrence \textit{regardless} of its degree. 
\begin{prop}\label{prop:SyndeticReturns}
    Suppose that \cite[Conjecture 2]{berMcCuIPPolySzemeredi} holds and let $H$ and $G$ be  countable  abelian groups. For any polynomial $p:H\rightarrow G$ with $p(0_H)=0_G$, any invertible probability preserving system $(X,\mathcal A,\mu,(T^g)_{g\in G})$, any $A\in\mathcal A$, and any $\epsilon>0$, the set 
    $$R_\epsilon^p(A):=\{h\in H\,|\,\mu(A\cap T^{-p(h)}A)>\mu^2(A)-\epsilon\}$$
    is syndetic, meaning that there is a non-empty finite subset $F\subseteq H$ such that $\bigcup_{h\in F}(R_\epsilon^p(A)+h)=H$.
\end{prop}
\begin{rem}\label{0.rem:FK}
When $\deg(p)=1$ and $p$ vanishes at zero,  every set of the form $R_\epsilon^p(A)$ is syndetic. See  \cite[p. 34]{MultifariousPoincare} for a simple proof. 
\end{rem}
\subsection{Statement of our main result}
 Given a countable abelian group $G$ and a probability preserving system $(X,\mathcal A,\mu,(T^g)_{g\in G})$, we say that $(T^g)_{g\in G}$ is weakly mixing if for every $A\in\mathcal A$ there exists a sequence $(g_k)_{k\in\N}$ in $G$ with the property that for any $B\in\mathcal A$, 
$$
\lim_{k\rightarrow\infty}\mu(A\cap T^{g_k}B)=\mu(A)\mu(B).
$$
For other equivalent definitions of a weakly mixing system see \cite[Theorem 4.1]{BerRos1988-MixingForGroups} (see also \cite[Theorem 1.8]{BerGor2005-WeakMixing}). Recall that  $\mathbb{F}_2$ denotes the field with two elements and $\FF$ is the vector space over $\mathbb{F}_2$ with countably infinite dimension. 
\begin{namedthm*}{Theorem A}\hypertarget{0.thm:FailureOfPolynomialKhintchine}
Let $d\in\N$. There is a  weakly mixing probability measure preserving system  
$$(X,\mathcal A,\mu,(T^\xi)_{\xi\in\mathbb F_2^\omega}
)$$
and a collection of sets $A_\delta\in\mathcal A$ with $\mu(A_\delta)=\delta$, $\delta\in (0,1/2]$, such that for every $\epsilon\in(0,1/2)$, there is a polynomial  $p_{d,\epsilon}:\FF\rightarrow \FF$ of degree $d$ mapping  $0_\FF$ to $0_\FF$ and with the property   that 
\begin{equation}\label{0.eq:SmallIntersections}
\{\xi\in \mathbb F_2^\omega\,|\,|\mu(A_\delta\cap T^{-p_{d,\epsilon}(\xi)}A_\delta)-\frac{\delta}{2^{d}}|\geq \epsilon\cdot\delta\}=\{0_\FF\}
\end{equation} 
for each $\delta\in (0,1/2]$. In fact, there is a polynomial $p_d:\FF\rightarrow\FF$ of degree $d$ with the property that for every $\epsilon\in (0,1/2)$ one can find an injective homomorphism $\varphi_{d,\epsilon}:\FF\rightarrow \FF$ such that $p_{d,\epsilon}=p_d\circ\varphi_{d,\epsilon}$.
\end{namedthm*}
\begin{rem}
Notice that the injective homomorphisms of the form $\varphi_{d,\epsilon}$ encode the following asymptotic feature of Theorem \hyperlink{0.thm:FailureOfPolynomialKhintchine}{A}:
    \begin{equation}\label{0.eq:IntuitiveIP*Limit}
\lim_{\epsilon\rightarrow 0^{+}}\sup_{\xi\in H_{d,\epsilon},\,\xi\neq 0_\FF}|\mu(A_\delta\cap T^{-p_d(\xi)}A_\delta)-\frac{\delta}{2^d}|=0,
    \end{equation}
where for each $\epsilon\in (0,1/2)$, $H_{d,\epsilon}$ is the infinite subgroup $\varphi_{d,\epsilon}(\FF)$ of $\FF$. As we will see in Example \ref{0.ExampleOfConvergence} below, the construction underlying Theorem \hyperlink{0.thm:FailureOfPolynomialKhintchine}{A} possesses related but somewhat stronger asymptotic properties which can be formalized by employing the language of IP$^*$-limits.
\end{rem}
\begin{rem}
    For an alternative formulation of Theorem \hyperlink{0.thm:FailureOfPolynomialKhintchine}{A} in which $p_{d,\epsilon}=p_d$ in formula \eqref{0.eq:SmallIntersections} for each $\epsilon>0$ (at the cost of modifying $(T^\xi)_{\xi\in\mathbb F_2^\omega}$), see Section \ref{6.Sec}.
\end{rem}
Next we state Corollary \hyperlink{0.Cor:MainResult}{B} which disproves \cite[Conjecture 2]{berMcCuIPPolySzemeredi} by showing that Corollary \ref{0.cor:FScor} cannot be extended to arbitrary countably infinite abelian groups. To deduce Corollary \hyperlink{0.Cor:MainResult}{B}, apply  Theorem \hyperlink{0.thm:FailureOfPolynomialKhintchine}{A} with $d\geq 2$, $\delta=1/2$, and $\epsilon=1/8$. Then for every non-zero $\xi\in\FF$,  
$$\mu(A_{1/2}\cap T^{-p_{d,1/8}(\xi)}A_{1/2})<\frac{1}{2^{d+1}}+\frac{1}{16}\leq \frac{3}{16}=\mu^2(A_{1/2})-\frac{1}{16}.$$
Hence $v_d:=p_{d,1/8}$ is not good for nice recurrence. 
\begin{namedthm*}{Corollary B}\hypertarget{0.Cor:MainResult}
For every $d>1$ there is a polynomial  $v_d:\mathbb F_2^\omega\rightarrow \mathbb F_2^\omega$ of degree $d$ vanishing at zero which is not good for nice recurrence. 
\end{namedthm*}
\subsection{The "sharpness" of Theorem A}
The result below establishes that, in a certain sense, Theorem \hyperlink{0.thm:FailureOfPolynomialKhintchine}{A} is "sharp". 
For each $d\in\N$, we let $a_d$ be the largest positive number with the property that for every invertible probability preserving system $(X,\mathcal A,\mu, (T^\xi)_{\xi\in\FF})$, every set $A\in\mathcal A$ with $\mu(A)=\frac{1}{2}$, and every polynomial  $p:\FF\rightarrow \FF$ of degree $d$ with $p(0_\FF)=0_\FF$, one has
$$
\sup_{\xi\in\FF,\,\xi\neq0_\FF}\mu(A\cap T^{-p(\xi)}A)\geq a_d.
$$
Thus $a_d$ is the best universal lower bound for degree $d$ polynomial-recurrence when $\mu(A)=1/2$.
\begin{namedthm*}{Corollary C}\hypertarget{0.prop:CorrectRateOfGrowth}
    For every $d\in\N$, 
    \begin{equation}\label{0.eq:SharpnessPoly}
    \frac{1}{2^{d+1}}\left(\frac{1}{2-1/2^d}\right)\leq a_d\leq \frac{1}{2^{d+1}}.
    \end{equation}
    Thus, $(a_d)_{d\in\N}$ is $O(\frac{1}{2^{d+1}})$.
\end{namedthm*}
It is worth pointing out that the lower bound of \eqref{0.eq:SharpnessPoly} establishes that for every $d\in\N$, every invertible probability preserving system $(X,\mathcal A,\mu, (T^\xi)_{\xi\in\FF})$, every $A\in\mathcal A$ with $\mu(A)\geq\frac{1}{2}-\frac{1}{2^{d+3}}$, and any polynomial $p:\FF\rightarrow\FF$ of degree at most $d$ vanishing at zero,  one has that for some non-zero $\xi\in\FF$,
\begin{equation}\label{0.eq:LargeIntersectionForDegreed}
\mu(A\cap T^{-p(\xi)}A)> \frac{1}{2^{2d+3}}.
\end{equation}
As a matter of fact, because every infinite subgroup of $\FF$ is isomorphic to $\FF$ and for any injective homomorphism $\varphi:\FF\rightarrow\FF$, $p\circ\varphi$ is a polynomial of degree at most $d$, it follows from  \eqref{0.eq:LargeIntersectionForDegreed}  that the set 
\begin{equation}\label{0.eq:LargeSyndeticReturns}
R^p(A):=\{\xi\in\FF\,|\,\mu(A\cap T^{-p(\xi)}A)>\frac{1}{2^{2d+3}}\}
\end{equation}
has  a non-trivial intersection with every infinite subgroup of $\FF$. So, in particular, it is syndetic. 
In the next subsection we will employ this observation to derive a combinatorial application of Theorem \hyperlink{0.thm:FailureOfPolynomialKhintchine}{A}.
\begin{rem}
    Combining \eqref{0.eq:LargeIntersectionForDegreed} with our results in Appendix \ref{C.Sec}, one obtains the following special case of the symmetric-difference weakening of DPHJ mentioned at the beginning of this Introduction: 
    \begin{adjustwidth}{2em}{2em}
Let $d\in\N$ and  let $\delta=\frac{1}{2}-\frac{1}{2^{d+3}}$. There is an $N\in\N$ such that if $\mathcal S\subseteq \mathcal P(\{1,...,N\}^d)$ satisfies $|\mathcal S|\geq \delta 2^{N^d}$, then one can find $E,F\in\mathcal S$ and a non-empty $\gamma\subseteq \{1,...,N\}$ with $E\triangle F=\gamma^d$.
\end{adjustwidth}
\end{rem}
\subsection{A combinatorial application of Theorem A}\label{1.6:Sec}
Before stating the combinatorial application of Theorem \hyperlink{0.thm:FailureOfPolynomialKhintchine}{A} mentioned above, we need to introduce some notation.\\

Let $(G,+)$ be a countable abelian group. Recall that a  sequence $(\Phi_N)_{N\in\N}$ of non-empty, finite  subsets  of $G$ is called a F{\o}lner sequence if for any $g\in G$,
$$
\lim_{N\rightarrow\infty}\frac{|\Phi_N\cap(\Phi_N-g)|}{|\Phi_N|}=1.
$$
For any $E\subseteq G$, we define the upper Banach density of $E$, denoted $d^*(E)$, by 
$$
d^*(E)=\sup_{\text{F{\o}lner sequences } (\Phi_N)}\limsup_{N\rightarrow\infty}\frac{|E\cap \Phi_N|}{|\Phi_N|}.
$$
We say that an element $g\in G$ 
 is  an \textit{$\epsilon$-popular difference} (for the set $E$) if  
 $$d^*(E\cap (E-g))\geq (d^*(E))^2-\epsilon.$$
 Equivalently, $g$ is an $\epsilon$-popular difference if there is a F{\o}lner sequence $(\Phi_N)_{N\in\N}$ in $G$ such that 
 $$
\limsup_{N\rightarrow\infty}\frac{|\{(h,h')\in (E\cap \Phi_N)^2\,|\,h-h'=g\}|}{|\Phi_N|}\geq(d^*(E))^2-\epsilon.
 $$ 
Since the polynomial $p_d$ defined in Theorem \hyperlink{0.thm:FailureOfPolynomialKhintchine}{A} has degree $d$,  Furstenberg's correspondence principle and formula \eqref{0.eq:LargeSyndeticReturns}
imply that for any $E\subseteq \FF$ with $d^*(E)\geq\frac{1}{2}-\frac{1}{2^{d+3}}$, the set 
$$
\{\xi\in\FF\,|\,p_d(\xi)\in E-E\}
$$
 is syndetic. The next result shows that this qualitative recurrence is compatible with the complete absence of non-trivial popular differences along $p_d(\FF)$.
 \begin{namedthm*}{Corollary D}\hypertarget{0.cor:CombinatorialApplication2}
    Let $d>1$ and let $p_d$ be as in the statement of Theorem \hyperlink{0.thm:FailureOfPolynomialKhintchine}{A}. There is  a set $E\subseteq \FF$ with $d^*(E)\geq \frac{1}{2}-\frac{1}{2^{d+3}}$ and an $\epsilon>0$, such that for every non-zero $\xi\in \FF$, $p_d(\xi)$ is not an $\epsilon$-popular difference for $E$. 
\end{namedthm*}
Several facts make  Corollary \hyperlink{0.cor:CombinatorialApplication2}{D}  somewhat surprising:
\begin{itemize}
\item Remark \ref{0.rem:FK} implies that for any countable abelian group $G$, any polynomial $p:\FF\rightarrow G$ vanishing at zero of degree one, any set $E\subseteq G$ with $d^*(E)>0$, and any $\epsilon>0$, the set 
    $$\{\xi\in\FF\,|\,p(\xi)\text{ is an $\epsilon$-popular difference for }E\}$$
    is syndetic.
\item  As shown in \cite[Theorem 4.13]{AckBer2025RingsOfIntegers}, there are many countable abelian groups $(G,+)$ with the property that for any member $p:G\rightarrow G$ of a certain class of polynomials, one has that if for every $E\subseteq G$ with $d^*(E)>0$, there are distinct $x,y\in E$ and a $g\in G$ with $p(g)=x-y$, then for every $E\subseteq G$ with $d^*(E)>0$ and every $\epsilon>0$, the set 
  $$\{g\in G\,|\,p(g)\text{ is an $\epsilon$-popular difference for }E\}$$
is syndetic.
\item  The results of \cite{BHKNilSystems2005} show that, in the context of multiple recurrence and measure preserving $\Z$-actions, the existence of measure-theoretical examples that fail to exhibit nice recurrence (see \cite[Theorem 2.1]{BHKNilSystems2005}) does not always translate into the existence of subsets of $\Z$ without $\epsilon$-popular differences (see \cite[Corollary 1.5]{BHKNilSystems2005}). This phenomenon was shown to hold for more general abelian groups in \cite{ABBLargeIntersection2021}.
\end{itemize}
\subsection{Background on VIP-systems and {\rm IP$^*$}-convergence}\label{1.5:Sec}
As we mentioned at the beginning of the Introduction, many of the results associated with  Theorem \hyperlink{0.thm:FailureOfPolynomialKhintchine}{A} are relevant not only for polynomial maps but also for a wider class of polynomial-like maps known as VIP-systems. 
In fact, much of the recurrence-theoretic structure surrounding Theorem \hyperlink{0.thm:FailureOfPolynomialKhintchine}{A} can be succinctly expressed with the help of IP$^*$-limits. 
In this subsection we review the necessary background on VIP-systems and IP-/IP$^*$-convergence. We remark that the material presented in this subsection will be useful throughout this work and  is necessary to state \cite[Conjecture 2]{berMcCuIPPolySzemeredi} in its original form, 
and hence to fully justify Proposition \ref{prop:SyndeticReturns} and the main claim of this paper.
\subsubsection{VIP-systems}
In \cite{BFM}, V. Bergelson, H. Furstenberg, and R. McCutcheon introduced the notion of a VIP-system which generalizes that of a polynomial mapping $p:H\rightarrow G$ satisfying $p(0_H)=0_G$. Let $\mathcal F$ denote the class of all non-empty finite subsets of $\N$ and let $G$ be an abelian group. A map $v:\mathcal F\rightarrow G$ is called a VIP-system if there exists a $t\in\N$ with the property that for any pairwise disjoint $\alpha_0,...,\alpha_t\in \mathcal F$,
\begin{equation}\label{0.eq:DefnVIP}
    \sum_{\beta\subseteq\{0,...,t\},\,\beta\neq\emptyset} (-1)^{|\beta|}v(\bigcup_{j\in\beta}\alpha_j)=0_G,
\end{equation}
where $|\beta|$ denotes the cardinality of $\beta$. We say that the degree of $v$ is $d\in\N$, denoted $\deg(v)$, if $d$ is the least $t\in\N$ for which \eqref{0.eq:DefnVIP} holds. When a VIP-system $v:\mathcal F\rightarrow G$ has degree 1, we call it  an IP-system. Notice that in this case one has that the sequence $g_k=v(\{k\})$, $k\in\N$, has the property that for any $\alpha\in\mathcal F$,
$$
g_\alpha:=\sum_{j\in\alpha}g_j=v(\alpha).
$$
The following observation relates polynomial mappings to VIP-systems: for any polynomial mapping $p:H\rightarrow G$ of degree $d$ satisfying $p(0_H)=0_G$ and any sequence $(h_k)_{k\in\N}$ in $H$, one has that the map $v:\mathcal F\rightarrow G$ given by
$v(\alpha)=p(h_\alpha)$
is a VIP-system of degree at most $d$ which we call the \textit{IP-polynomial generated by $p$ and $(h_\alpha)_{\alpha\in\mathcal F}$}. (To see why $\deg(v)\leq d$, it is sufficient to note that for any  $y_0,...,y_d\in H$, 
\begin{multline*}
0_G=D_{y_d}\cdots D_{y_0}p(0_H)=\sum_{\beta\subseteq\{0,...,d\},\,\beta\neq\emptyset} (-1)^{(d+1)-|\beta|}p(y_\beta)+(-1)^{d+1}p(0_H)\\
=(-1)^{(d+1)}\sum_{\beta\subseteq\{0,...,d\},\,\beta\neq\emptyset} (-1)^{|\beta|}p(y_\beta).)
\end{multline*}
Moreover, as shown in \cite[Theorem 1.2]{BHM}, not all VIP-systems can be written as the composition of a polynomial mapping with an IP-system. (In other words, not every VIP-system is an IP-polynomial.)\\
Next we present  a "simplified", equivalent form of \cite[Conjecture 2]{berMcCuIPPolySzemeredi} that does not involve the use of IP- or IP$^*$-convergence. We will present (a slight variant of) the original statement of \cite[Conjecture 2]{berMcCuIPPolySzemeredi} below. 
\begin{conjecture}\label{0.Conjecture A}
    Let $(G,+)$ be a countable abelian group and let $v:\mathcal F\rightarrow G$ be a VIP-system. Then, $v$ is good for nice recurrence. 
\end{conjecture}
\begin{rem}
    Observe that Proposition \ref{prop:SyndeticReturns} is an immediate consequence of  Conjecture \ref{0.Conjecture A}. Indeed, consider abelian groups $H$ and $G$ and let $p:H\rightarrow G$ be a polynomial mapping vanishing at zero. By Conjecture \ref{0.Conjecture A}, one has that for every IP-system $(h_\alpha)_{\alpha\in\mathcal F}$ in $H$ and any set of the form $R_\epsilon^p(A)$, there is an $\alpha\in \mathcal F$, with $h_\alpha\in R_\epsilon^p(A)$. Since any subset of $H$ having a non-trivial intersection with every set of finite sums $\{h_\alpha\,|\,\alpha\in\mathcal F\}$ is syndetic, it follows that $R_\epsilon^p(A)$ is syndetic. 
\end{rem}
We conclude our discussion of VIP-systems by noting that they can be naturally interpreted as polynomial-like maps whose domain is $\FF$. As we will see, this interpretation turns out to be particularly well-suited for the flow of our discussion. \\
It is well-known that  $(\FF,+)$ is isomorphic to $(\mathcal F_\emptyset,\triangle)$, where $\mathcal F_\emptyset:=\mathcal F\cup\{\emptyset\}$ and $\triangle$ denotes the symmetric difference operator.
This isomorphism, which will be used implicitly throughout the paper, can be described as follows:
Identify  $\FF$ with the set 
$$
\{\xi\in\{0,1\}^\N\,|\,\exists N\in\N\forall n\geq N,\,\xi(n)=0\}.
$$
Under this identification for any  $\xi,\eta\in\FF$ and any $n\in\N$, 
$$(\xi+\eta)(n)\equiv \xi(n)+\eta(n)\mod 2.$$
It follows that the map 
$$\xi\mapsto \alpha_\xi:=\{n\in\N\,|\,\xi(n)\neq 0\}$$
is the desired group isomorphism.
Under this identification, one has that for any given abelian group $(G,+)$, polynomial mappings $p:\FF\rightarrow G$ and $G$-valued VIP-systems are related in two simple ways:
\begin{itemize}
    \item [-] The restriction of any  polynomial mapping $p:\FF\rightarrow G$ with $p(0_\FF)=0_G$ to $\mathcal F=\FF\setminus\{0_\FF\}$ is a VIP-system. Notice that the degree of $p$ viewed as a VIP-system  is at most that of $p$ when viewed as a polynomial.
    \item [-] For any VIP-system $v:\mathcal F\rightarrow G$ one can extend the domain of $v$ to $\mathcal F_\emptyset$ by setting $v(\emptyset)=0_G$. Thus $v$ may be viewed as a map from $\FF$ to $G$ with properties similar to those of a  regular polynomial (although, $v$ may not satisfy \eqref{0.eq:DefnDP}). 
\end{itemize}
\subsubsection{{\rm IP$^*$}-convergence}
The notion of IP-limit  plays a central role in stating and proving key results in IP-ergodic theory (see, for example, \cite{Fbook}, \cite{FKIPSzemerediLong}, \cite{BFM}, \cite{mccutcheon2005fvip}, \cite{ZorinIPNilpotentSz}). We next  define  a variant of this notion of convergence dealing with $\FF$-indexed sequences and which corresponds to the classical notion of  IP$^*$-limit introduced by Furstenberg in \cite{Fbook}. As we will see, IP-limits and IP$^*$-limits of $\FF$-indexed sequences are, so to say, \textit{quasi-equivalent} (see Remark \ref{0.rem"IPvsIP*Lim} below).\\
Let $(X,d)$ be a metric space and let $(x_\xi)_{\xi\in\FF}$ be an $\FF$-indexed sequence. We write
$$
\mathop{\text{{\rm IP$^*$-lim}}}_{\xi\in\FF}\; x_\xi=x
$$
if there is an $x\in X$ with the property that for every infinite subgroup $H$ of $\FF$ and every $\epsilon>0$ there is a non-zero $\xi\in H$ such that
$d(x,x_\xi)<\epsilon$. The following result due to N. Hindman \cite{HIPPartitionRegular} implies that whenever $X$ is a compact metric space, one can always find an injective homomorphism $\varphi:\FF\rightarrow \FF$ and an $x\in X$ such that 
$
\mathop{\text{{\rm IP$^*$-lim}}}_{\xi\in\FF}\; x_{\varphi(\xi)}=x.
$ Here we use the elementary observation that  any infinite subgroup of $\FF$ is isomorphic to $\FF$.
\begin{thm}[Hindman's Theorem on $\FF$]\label{0.thm:Hindman}
Let $r\in\N$ and let $C_1,...,C_r\subseteq \FF$ be such that $\FF=\bigcup_{j=1}^rC_j$. Then, there is a $j_0\in\{1,...,r\}$ and an infinite subgroup $H$ of $\FF$ such that $H\subseteq C_{j_0}\cup \{0_\FF\}$. 
\end{thm}
\begin{rem}\label{0.rem"IPvsIP*Lim}
Let $(X,d)$ be a compact metric space and let $(x_\xi)_{\xi\in\FF}$ be a sequence in $X$. 
    Recall that 
$\mathop{\text{{\rm IP-lim}}}_{\xi\in\FF}\; x_{\xi}=x
$ if as $\min \alpha_\xi\to\infty$, one has that $x_\xi\to x$. We now explain the sense in which IP-limits and IP$^*$-limits are quasi-equivalent:
\begin{enumerate}
    \item [-] If $\mathop{\text{{\rm IP-lim}}}_{\xi\in\FF}\; x_{\xi}=x
$, then $\mathop{\text{{\rm IP$^*$-lim}}}_{\xi\in\FF}\; x_{\xi}=x
$.
\item [-] On the other hand, if  $\mathop{\text{{\rm IP$^*$-lim}}}_{\xi\in\FF}\; x_{\xi}=x
$, one has that for every injective homomorphism $\varphi:\FF\rightarrow\FF$ for which $\mathop{\text{{\rm IP-lim}}}_{\xi\in\FF}\; x_{\varphi(\xi)}$ exists,  the  limit must be $x$. (Notice that because $X$ is compact, Hindman's theorem implies that at least one such homomorphism exists.)
\end{enumerate}
\end{rem}
\begin{example}\label{0.ExampleOfConvergence}
Let $d\in\N$ and let $(X,\mathcal A,\mu,(T^\xi)_{\xi\in\mathbb F_2^\omega}
)$, $p_d$, and $A_\delta\in\mathcal A$ with $\mu(A_\delta)=\delta$, $\delta\in (0,1/2]$, be as in the statement of Theorem \hyperlink{0.thm:FailureOfPolynomialKhintchine}{A}. One can show that for every $\delta\in (0,1/2]$, 
\begin{equation}\label{0.eq:IP*LimitThmA}
    \mathop{\text{{\rm IP$^*$-lim}}}_{\xi\in\FF}\,\mu(A_\delta\cap T^{-p_d(\xi)}A_\delta)=\frac{\delta}{2^d}.
\end{equation}
Indeed, our proof of Theorem \hyperlink{0.thm:FailureOfPolynomialKhintchine}{A} implies that
$$
\lim_{|\alpha_\xi|\rightarrow\infty}\mu(A_\delta\cap T^{-p_d(\xi)}A_\delta)=\frac{\delta}{2^d}.
$$
So, since for any injective homomorphism $\varphi:\FF\rightarrow\FF$ and any $N\in\N$, one can always find a $\xi\in\FF$ with 
$|\alpha_{\varphi(\xi)}|\geq N$, it follows that \eqref{0.eq:IP*LimitThmA} holds.\\
Notice that Hindman's theorem implies that for every $\epsilon>0$ there is an injective $\varphi:\FF\rightarrow\FF$ with the property that for every non-zero $\xi\in\FF$, $$|\mu(A_\delta\cap T^{-p_d(\varphi(\xi))}A_\delta)-\frac{\delta}{2^d}|<\epsilon\cdot\delta.$$
Thus, the IP$^*$-limit statement \eqref{0.eq:IP*LimitThmA}, combined with Hindman's theorem, guarantees the existence of an injective homomorphism $\tilde\varphi_{d,\epsilon}$ along which the estimate in  \eqref{0.eq:SmallIntersections} holds uniformly for all non-zero $\xi\in\FF$, although the argument presented here may produce a homomorphism depending also on $\delta$.
\end{example}
We are now in position to state  a version of \cite[Conjecture 2]{berMcCuIPPolySzemeredi} adapted to the terminology of this paper.
 \begin{conjecture}[Conjecture 2 in \cite{berMcCuIPPolySzemeredi}]\label{0.Conjecture2}
Let $\mathcal H$ be a Hilbert space and let $(V_\xi)_{\xi\in\FF}$ be a VIP-system of unitary operators acting on $\mathcal H$. Suppose that for every $f\in \mathcal H$, there is a $Qf\in\mathcal H$ with  
$$\mathop{\text{{\rm IP-lim}}}_{\xi\in\FF} V_{\xi}f=Qf$$
weakly. Then, $\Re(\langle f,Qf\rangle)\geq 0$ or,  equivalently, 
$$
\mathop{\text{{\rm IP$^*$-lim}}}_{\xi\in\FF} \Re(\langle f, V_{\xi}f\rangle) \geq 0.
$$
 \end{conjecture}
In Appendix \ref{B.Sec} we will show that Conjectures \ref{0.Conjecture A} and \ref{0.Conjecture2} are equivalent. While the fact that Conjecture \ref{0.Conjecture2} implies Conjecture \ref{0.Conjecture A} is transparent, the other direction will make use of a sharp condition connecting nice recurrence for Gaussian systems and measurable sets of measure $1/2$ with arbitrary positive definite sequences.  
\subsection{Connections with the density polynomial Hales-Jewett conjecture}\label{0.sec:dPHJ}
Theorem \hyperlink{0.thm:FailureOfPolynomialKhintchine}{A} shows that Corollary \ref{0.cor:FScor} does not extend to arbitrary countably infinite abelian groups. In particular, Theorem \hyperlink{0.thm:FailureOfPolynomialKhintchine}{A} suggests that the techniques used to establish variants of \cref{0.FSTheorem} in \cite{BFM}, \cite{BHM}, \cite{BDonaldRobertsonIP_r}, \cite{McCutWSarkozy2014}, \cite{AckBer2025RingsOfIntegers} do not suffice to extend the Furstenberg-S{\'a}rk{\"o}zy theorem to arbitrary countable abelian groups. 
In this subsection we discuss the type of positive results that one could still expect despite the consequences of Theorem \hyperlink{0.thm:FailureOfPolynomialKhintchine}{A}. 
We  proceed in three steps: 
First, we record how a suitable variant of the density polynomial Hales-Jewett conjecture (DPHJ) would imply a result involving recurrence along (not necessarily $\FF$-valued)  VIP-systems. 
Then, we discuss  how Theorem \hyperlink{0.thm:FailureOfPolynomialKhintchine}{A} constrains this recurrence result. 
Finally, we show that the methods we employ to prove Theorem \hyperlink{0.thm:FailureOfPolynomialKhintchine}{A} cannot be easily modified to disprove DPHJ. 
\subsubsection{A consequence of DPHJ}
 We begin by stating the following restricted form of the density polynomial Hales-Jewett conjecture.
\begin{conjecture}[Cf. Conjecture 4 in  \cite{GowersDPHJjBlog}]\label{0.DPHJSpecialCase}
    Let $d\in\N$ and let $\delta\in(0,1)$. There exists  an $N:=N(d,\delta)$ in $\N$ with the property that for any $S\subseteq \mathcal P(\{1,...,N\}^{d})$ with $|S|>\delta2^{N^d}$ there exist an $A\in S$ and a non-empty $\gamma\subseteq \{1,...,N\}$ with the properties that $A\cap \gamma^d=\emptyset$ and $A\cup \gamma^d\in S$.
\end{conjecture}
\begin{rem}
The density polynomial Hales-Jewett conjecture (DPHJ) is widely regarded as one of the central open problems in Ramsey theory. It was first proposed by V. Bergelson in \cite[p. 56]{ERTaU} and is expected to refine the Polynomial Hales-Jewett theorem \cite{BerLeibPolyHJ} in much the same way that the density Hales-Jewett theorem of H. Furstenberg and Y. Katznelson \cite{FurKatzDHJ} strengthens the classical Hales-Jewett theorem \cite{HalesJewett1963}.
\end{rem}
We next state the conditional strengthened version of the Furstenberg-S{\'a}rk{\"o}zy theorem mentioned above. It can be deduced from Conjecture \ref{0.DPHJSpecialCase} in a way similar to the one in which (i)$\implies$(ii) in Theorem \ref{C.thm:EquivalentForms} is proved (see also \cite[Theorem 6.15]{berMcCuIPPolySzemeredi}, \cite[Appendix A]{BerZel-NiceRecurrence}). 
\begin{prop}\label[proposition]{0.prop:ErgodicConsequenceOfDPHJ}\label{0.prop.StrongRec}
   Let $d\in\N$ and $\delta\in (0,1]$. Suppose that there is a $t\in(0,1)$ for which Conjecture \ref{0.DPHJSpecialCase} holds for the pair $(d,(1-t)\delta)$. Then, there is a constant $c_{d,\delta}>0$ with the property that for any countable  abelian group $G$, any probability preserving system $(X,\mathcal A,\mu,(T^g)_{g\in G})$, any $A\in\mathcal A$ with $\mu(A)\geq \delta$, any $\epsilon>0$, and any VIP-system $v:\FF\rightarrow G$ of degree at most $d$, there is a non-zero $\xi_\epsilon\in\FF$ with 
    \begin{equation*}
\mu(A\cap T^{-v(\xi_\epsilon)}A)> c_{d,\delta}-\epsilon.
    \end{equation*}
Equivalently, for any injective homomorphism $\varphi:\FF\rightarrow\FF$ for which $\mathop{\text{{\rm IP$^*$-lim}}}_{\xi\in\FF}\mu(A\cap T^{-v(\varphi(\xi))}A)$ exists, one has that  
\begin{equation}\label{0.eq:DPHJLowerBound}
\mathop{\text{{\rm IP$^*$-lim}}}_{\xi\in\FF}\mu(A\cap T^{-v(\varphi(\xi))}A)\geq c_{d,\delta}.
\end{equation}
\end{prop}
\subsubsection{Constraints on the constants of the form $c_{d,\delta}$ that follow from Theorem A} 
Notice that \cite[Conjecture 2]{berMcCuIPPolySzemeredi} predicted that
for every $d$ and every $\delta$, the constant  $c_{d,\delta}$ in formula \eqref{0.eq:DPHJLowerBound} can be taken to  equal $\delta^2$ (regardless of the value of $d$).  Theorem \hyperlink{0.thm:FailureOfPolynomialKhintchine}{A} shows  that this cannot be the case. In fact, Theorem \hyperlink{0.thm:FailureOfPolynomialKhintchine}{A} implies that for any
$\delta\in (0,1/2]$ and any choice of $(c_{d,\delta})_{d\in\N}$, one necessarily has 
\begin{equation}\label{0.eq:LimitsAre0}
\lim_{d\rightarrow\infty}c_{d,\delta}=0.
\end{equation}
Observe that, in principle, formula \eqref{0.eq:LimitsAre0} is compatible with the assertion that for every VIP-system $v$ taking values in some abelian group $G$, one always has that for every probability preserving system $(X,\mathcal A,\mu,(T^g)_{g\in G})$, and every $A\in\mathcal A$ with $\mu(A)>0$, the set 
$$
R^{v}(A):=\{\xi\in\FF\,|\,\mu(A\cap T^{-v(\xi)}A)>0\}
$$
is syndetic. 
As a matter of fact, the extension of Corollary \hyperlink{0.prop:CorrectRateOfGrowth}{C} proved in Section \ref{5.Sec}
confirms that the coexistence of these phenomena occurs at the very least in the case where $G=\FF$ and $\mu(A)=1/2$. However, one may be tempted to interpret formula \eqref{0.eq:LimitsAre0} as an omen of the existence of a VIP-system $v$  for which a set of the form $R^v(A)$ fails to be syndetic, disproving DPHJ.\\
When we restrict our attention to the case when $\mu(A)=1/2$ and look at the functions of the form 
$$f_A=2(\mathbbm 1_A-\mu(A)),$$
one sees that  confirming that there is no such  $v$ is tantamount to answering negatively  the following question (as we do below).
\begin{question}\label{0.Q:NonRecurrence?}
    Is there a VIP-system $(V_\xi)_{\xi\in\FF}$   taking values in an abelian group of unitary operators acting on a Hilbert space $\mathcal H$ and  an  $f\in\mathcal H$ with $\|f\|=1$, for which
$$
\mathop{\text{{\rm IP$^*$-lim}}}_{\xi\in\FF}\;\Re(\langle f,V_{\xi}f\rangle)=-1?
$$
\end{question}
\subsubsection{VIP-recurrence for general abelian groups}
The result below demonstrates that the conclusion of Proposition \ref{0.prop.StrongRec} does hold when $\delta=1/2$ and therefore provides a negative answer to Question \ref{0.Q:NonRecurrence?}. 
 We remark that our proof of Proposition \hyperlink{0.prop:UnitaryLowerBound}{E} makes use of a highly restricted version of the density polynomial Hales-Jewett conjecture which we prove in Appendix \ref{A.Sec}.
\begin{namedthm*}{Proposition E}\hypertarget{0.prop:UnitaryLowerBound}
    Let $d\in\N$. There exists an $\epsilon_d>0$ with the property that  for any VIP-system $(V_\xi)_{\xi\in\FF}$ of degree at most $d$  taking values in an abelian group of unitary operators acting on a Hilbert space $\mathcal H$,  any $f\in\mathcal H$, and any injective homomorphism $\varphi:\FF\rightarrow \FF$ for which 
    $
\mathop{\text{{\rm IP$^*$-lim}}}_{\xi\in\FF}\;\langle f,V_{\varphi(\xi)}f\rangle
$
exists,
one has that 
\begin{equation}\label{0.eq:PropositionE}
\mathop{\text{{\rm IP$^*$-lim}}}_{\xi\in\FF}\;\Re(\langle f,V_{\varphi(\xi)}f\rangle)\geq \|f\|_\mathcal H^2(-1+\epsilon_d).
\end{equation}
\end{namedthm*}
\begin{rem}\label{0.VIPficationRemark}
This remark deals with the VIP-structure of the expression $(V_{\varphi(\xi)})_{\xi\in\FF}$ appearing in \eqref{0.eq:PropositionE}.
   Let $(G,+)$ be an abelian group, let $v:\FF\rightarrow G$ be a VIP-system of degree $d$, and let $\varphi:\FF\rightarrow\FF$ be an injective homomorphism. Notice  that if $\xi,\eta\in\FF\setminus\{0_\FF\}$ satisfy $\alpha_{\xi}\cap\alpha_{\eta}=\emptyset$ it may happen that  $\alpha_{\varphi(\xi)}\cap\alpha_{\varphi(\eta)}\neq\emptyset$ and, so, $v\circ\varphi$ may fail to be a VIP-system. Nonetheless, one can always find an injective homomorphism $\psi:\FF\rightarrow \FF$ with the property that for any non-zero $\xi,\eta\in\FF$, if $\max \alpha_\xi<\min \alpha_\eta$, then $\max \alpha_{\psi(\xi)}<\min \alpha_{\psi(\eta)}$ and  $\max \alpha_{\varphi(\psi(\xi))}<\min \alpha_{\varphi(\psi(\eta))}$ (see Lemma \ref{B.lem:NiceInjectiveRestrictionForVIP} below).  Thus, the map $v\circ \varphi\circ \psi$ is a VIP-system of degree at most $d$ which, in view of the definition of IP$^*$-limit, has the additional property that if   $\mathop{\text{{\rm IP$^*$-lim}}}_{\xi\in\FF}\, x_{v(\varphi(\xi))}$ exists, then 
   \begin{equation}
\mathop{\text{{\rm IP$^*$-lim}}}_{\xi\in\FF}\, x_{v(\varphi(\xi))}=\mathop{\text{{\rm IP$^*$-lim}}}_{\xi\in\FF}\, x_{v(\varphi(\psi(\xi)))}
   \end{equation}
(Thus, any such IP$^*$-limit  can be \textit{computed} by looking at the "VIP-sequence" $v(\varphi(\psi(\xi))),$ $\xi\in\FF$).
\end{rem}
\begin{rem}
At first glance, the proof of Proposition \hyperlink{0.prop:UnitaryLowerBound}{E} may seem to require a substantially more sophisticated argument than that of  Corollary \hyperlink{0.prop:CorrectRateOfGrowth}{C}.
 After all, $\FF$ is one of the simplest (if not the simplest) infinitely generated abelian groups, and a general VIP-system leaves more room for pathological behavior than an $\FF$-valued polynomial. 
 However,  with sufficient bookkeeping and error control, the Fourier-theoretic proof of Corollary \hyperlink{0.prop:CorrectRateOfGrowth}{C} in Section \ref{5.Sec} can be adapted to prove Proposition \hyperlink{0.prop:UnitaryLowerBound}{E} directly. We use the restricted DPHJ statement proved in Appendix \ref{A.Sec} instead in order to make visible the  combinatorial mechanism behind the argument. 
\end{rem}

The structure of this paper is as follows: In Section \ref{2.Sec} we define the sequence of polynomial maps $p_d:\FF\rightarrow \FF$, $d\in\N$,  appearing in the statement of Theorem \hyperlink{0.thm:FailureOfPolynomialKhintchine}{A}. In Section 3 we prove several technical lemmas dealing with Markov processes. These lemmas will facilitate the necessary estimates in the subsequent  section. In Section 4 we prove Theorem \hyperlink{0.thm:FailureOfPolynomialKhintchine}{A}.
In Section 5 we prove an extension of Corollary  \hyperlink{0.prop:CorrectRateOfGrowth}{C} dealing with VIP-systems.
In Section 6 we prove  Corollary \hyperlink{0.cor:CombinatorialApplication2}{D} and deduce a variant of Theorem \hyperlink{0.thm:FailureOfPolynomialKhintchine}{A}.
In Appendix \ref{A.Sec} we prove the highly restricted variant of DPHJ needed for the proof of Proposition \hyperlink{0.prop:UnitaryLowerBound}{E}.
In Appendix \ref{B.Sec} we establish a general result which provides a correspondence between measurable recurrence phenomena and positive definite sequences and which allows us to prove the equivalence of  Conjectures \ref{0.Conjecture A} and \ref{0.Conjecture2} and to prove Proposition \hyperlink{0.prop:UnitaryLowerBound}{E}.
In Appendix \ref{C.Sec} we establish various equivalent forms of a variant of the Furstenberg-S{\'a}rk{\"o}zy theorem for $\FF$-valued polynomials of degree at most $d$ and sets of measure at least $\delta>0$. \\

\textit{Acknowledgments.} Rigoberto Zelada is supported by EPSRC through Joel Moreira's Frontier Research Guarantee grant, ref. EP/Y014030/1. 
\section{Defining the maps $p_d$, $d\in\N$}\label{2.Sec}
In this section we define the maps $p_d:\mathbb F_2^\omega\rightarrow \mathbb F_2^\omega$ mentioned in the statement of Theorem \hyperlink{0.thm:FailureOfPolynomialKhintchine}{A} and prove that there is a sequence of injective homomorphisms $\Phi_N:\FF\rightarrow \FF$, $N\in\N$, with the property that for every $d\in\N$ and every $N\in\N$, $p_d\circ \Phi_N$ is  a polynomial of degree $d$ satisfying  $(p_d\circ \Phi_N)(0_{\mathbb F_2^\omega})=0_{\mathbb F_2^\omega}$. Before defining the sequences  $(p_d)_{d\in\N }$ and $(\Phi_N)_{N\in\N}$, we need to introduce some notation.
\subsection{Convenient representations for the elements of $\mathbb F_2^\omega$.}
For each $d\in\N$, let $\mathcal I_d:\N\rightarrow \N^{d}$ be a bijection. Fix $d\in\N$ and let $(x_n^{(1)})_{n\in\N}$,...,$(x_n^{(d)})_{n\in\N}$ be sequences of $\{0,1\}$-valued variables.  For every $\xi\in\FF$ we identify $\xi$ with the polynomial 
$
p_{d,\xi}
$, the $\mathbb F_2$-valued polynomial in the variables $(x_n^{(j)})_{n\in\N}$, $j\in\{1,...,d\}$, defined by 
$$p_{d,\xi}:=\sum_{i\in\alpha_\xi}x_{\mathcal I_d(i,1)}^{(1)}\cdots x_{\mathcal I_d(i,d)}^{(d)},$$
where we adopt the conventions that 
$$\mathcal I_d(i)=(\mathcal I_d(i,1),...,\mathcal I_d(i,d))$$
and that 
$\sum_{i\in\emptyset}x_{\mathcal I_d(i,1)}^{(1)}\cdots x_{\mathcal I_d(i,d)}^{(d)}=0_{\mathbb F_2}$. Via the identification $\xi\mapsto p_{d,\xi}$,  we obtain that  $(\mathbb F_2^\omega,+)$ is isomorphic to $\mathcal P_{d}$, the set of polynomials of the form $\sum_{i\in\alpha}x_{\mathcal I_d(i,1)}^{(1)}\cdots x_{\mathcal I_d(i,d)}^{(d)}$, $\alpha\in\mathcal F_\emptyset$, with addition $\mod 2$.
\subsection{The maps $p_d:\FF\rightarrow \FF$, $d\in\N$}
In order to define the maps $p_d:\FF\rightarrow \FF$, $d\in\N$, we will first define the maps $q_d:\FF\rightarrow \mathcal P_d$, $d\in\N$. For each $d\in\N$, let $q_d(0_\FF)=0_{\mathcal P_d}$ and for every non-zero $\xi\in\FF$, let 
\begin{equation}\label{2.eq:Defnq_d}
q_d(\xi)= \prod_{t=1}^d(\sum_{i\in\alpha_\xi} x_i^{(t)}).
\end{equation}
Identifying first $\FF$ with $\FF\times \FF$ and then $\FF\times \FF$ with $\mathcal P_d\times \FF$, we define 
\begin{equation}\label{2.eq:Defnp_d}
p_d:\FF\rightarrow \FF\simeq \mathcal P_d\times \FF
\end{equation}
by $p_d(\xi)=(q_d(\xi),0_\FF)$. Notice that when we view  $p_d$ as a map from $\FF$ to $\FF\times\FF$, we have that for any non-zero $\xi\in\FF$, $p_d(\xi)=(\eta,0_\FF)$, where $\eta\in\FF$ satisfies 
\begin{equation}\label{2.eq:Defnp_d'}
\prod_{t=1}^d(\sum_{i\in\alpha_\xi} x_i^{(t)})=\sum_{i\in\alpha_\eta}x_{\mathcal I_d(i,1)}^{(1)}\cdots x_{\mathcal I_d(i,d)}^{(d)}.
\end{equation}
In other words, for every $\xi\in\FF$, $p_d(\xi)=(p_{d,\eta}, 0_\FF)$. 
\\
For every $N\in\N$ and every $\alpha\in\mathcal F_\emptyset$, we let 
\begin{equation}\label{2.eq:dAmplification}
\langle \alpha\rangle_N:=\bigcup_{s=0}^{N-1}\{jN-s\,|\,j\in\alpha\}=\bigcup_{j\in\alpha}\{(j-1)N+1,...,jN\}.
\end{equation}
So, when $\alpha=\emptyset$, $\langle \alpha\rangle_N =\emptyset$ and  when $N=1$, $\langle \alpha\rangle_1=\alpha$. For each $N\in\N$, we will let $\Phi_N:\FF\rightarrow \FF$ be the injective homomorphism given by $\alpha_{\Phi_N(\xi)}=\langle \alpha_\xi\rangle_N$. \\
The following result shows that for any $N\in\N$,  $p_d\circ \Phi_N$ is a polynomial of degree $d$.
\begin{lem}\label{2.p_dIsVIPofDegreeD}
    Let $d\in\N$ and let $N\in\N$. The map  $p_d\circ \Phi_N$ is a polynomial of degree $d$.
\end{lem}
\begin{proof}
Fix the natural numbers $d$ and $N$. In order to prove that $p_d\circ\Phi_N$ is a polynomial of degree $d$ it is sufficient to show that $q_d$ is a polynomial of degree not more than $d$ and that $q_d\circ\Phi_N$  has degree strictly larger than $d-1$. 
We will first show that $q_d$ is a polynomial of degree at most $d$.\\

$\bullet$ \textit{ The degree of $q_d$ is at most $d$.} Pick any $\xi,\xi_0,...,\xi_{d}\in \FF$. We will prove that 
$$D_{\xi_d}\cdots D_{\xi_0}q_d(\xi)=0_\FF.$$
To do this, first note that 
\begin{multline*}
D_{\xi_d}\cdots D_{\xi_0}q_d(\xi)\equiv\sum_{\beta\subseteq \{0,...,d\}}(-1)^{(d+1)-|\beta|}q_d(\xi+\sum_{j\in\beta}\xi_j)
\equiv\sum_{\beta\subseteq \{0,...,d\}}q_d(\xi+\sum_{j\in\beta}\xi_j)\\
\equiv\sum_{\beta\subseteq\{0,...,d\}}\prod_{t=1}^d(\sum_{i\in\alpha_\xi}x_i^{(t)}+\sum_{j\in\beta}\sum_{i\in\alpha_{\xi_j}}x_i^{(t)})\equiv\sum_{i_1,...,i_d\in \alpha_\xi\cup\bigcup_{j=0}^d\alpha_{\xi_j}}\gamma_{i_1,...,i_d}x_{i_1}^{(1)}\cdots x_{i_d}^{(d)}\mod 2,
\end{multline*}
where for each $i_1,...,i_d\in \alpha_\xi\cup\bigcup_{j=0}^d\alpha_{\xi_j}$, the constant $\gamma_{i_1,...,i_d}$ is given by 
$$
\gamma_{i_1,...,i_d}\equiv\sum_{\beta\subseteq\{0,...,d\}}\prod_{t=1}^d(\mathbbm 1_{\alpha_\xi}(i_t)+\sum_{j\in\beta}\mathbbm 1_{\alpha_{\xi_j}}(i_t))\mod 2.
$$
To complete the proof of the fact that $\deg(q_d)\leq d$, we will show that for any given $i_1,...,i_d\in \alpha_\xi\cup\bigcup_{j=0}^d\alpha_{\xi_j}$, $\gamma_{i_1,...,i_d}\equiv 0\mod 2$. For each $t\in\{1,...,d\}$, define the homomorphism $\sigma_t:\mathbb F_2^{d+1}\rightarrow \mathbb F_2$ by 
$$
\sigma_t(y_0,...,y_d)\equiv\sum_{j=0}^d y_j\mathbbm 1_{\alpha_{\xi_j}}(i_t)\mod 2
$$
and note that $\gamma_{i_1,...,i_d}\equiv 1\mod 2$ if and only if the number of solutions $(y_0,...,y_d)\in \mathbb F_2^{d+1}$ to the system of equations
\begin{equation}\label{eq:SystemOfEquations}
1\equiv \mathbbm 1_{\alpha_\xi}(i_t)+\sigma_t(y_0,...,y_d)\mod 2,\,t\in\{1,...,d\},
\end{equation}
is odd. 
Define the homomorphism $L:\mathbb F_2^{d+1}\rightarrow \mathbb F_2^d$ by 
$$L({\bf y})=(\sigma_1({\bf y}),...,\sigma_d({\bf y})).
$$
It follows that the set of solutions of \eqref{eq:SystemOfEquations}  coincides with 
$L^{-1}(\mathbbm 1_{\alpha_\xi}(i_1)+1,...,\mathbbm 1_{\alpha_\xi}(i_d)+1))$. Noting that $|\text{Ker}(L)|\cdot|L(\mathbb F^{d+1}_2)|=|\mathbb F^{d+1}_2|$, we obtain that the number of solutions to 
\eqref{eq:SystemOfEquations} must be even.\\

$\bullet$ \textit{ For any $N\in\N$, $q_d\circ \Phi_N$ has degree more than $d-1$.} 
Assume first that 
$d=1$. By  our definition of the degree of a group-theoretical polynomial, we must have $\deg(q_d\circ\Phi_N)\in\N$. It follows that $\deg(q_d\circ\Phi_N)=1$. Assume now that $d>1$, and consider any $n\in\{0,...,d-1\}$.
Let $\alpha_0,...,\alpha_n\subseteq\{1,...,d\}$ be non-empty, pairwise disjoint sets with the property that $\bigcup_{j=0}^n\alpha_j=\{1,...,d\}$. It follows that 
 for any $s_1,...,s_{d}\in\{0,...,N-1\}$, the term
\begin{equation}\label{2.eq:GoodMonomial}
\prod_{t=1}^d x_{tN-s_t}^{(t)}
\end{equation}
does not appear in any of the polynomials
$$q_d(\Phi_N(\alpha_{i_1}\triangle\cdots\triangle \alpha_{i_r})),\,0\leq i_1<\cdots< i_{r}\leq n\text{ and }r<n+1.$$
Furthermore, every term of the form \eqref{2.eq:GoodMonomial} appears exactly once in $$q_d(\Phi_N(\alpha_{0}\triangle\cdots\triangle \alpha_{n}))=q_d(\Phi_N(\{1,...,d\})).$$ Thus,  every term of the form \eqref{2.eq:GoodMonomial} appears exactly once in the expression 
\begin{equation}\label{2.eq:VIPExpression}
D_{\alpha_0}\cdots D_{\alpha_n}(q_d\circ \Phi_N)(0_\FF)\equiv\sum_{r=1}^{n+1}\sum_{0\leq i_1<\cdots< i_{r}\leq n}q_d(\Phi_N(\alpha_{i_1}\triangle\cdots\triangle \alpha_{i_r}))+q_d(\Phi_N(0_\FF))\mod 2,
\end{equation}
proving that \eqref{2.eq:VIPExpression} is non-zero and, so, that $q_d\circ \Phi_N$ has degree at least $d$.
We are done. 
\end{proof}
\begin{rem}\label{2.rem:VipDegOFp_d}
    Let $d,N\in\N$. Notice that our proof of  Lemma \ref{2.p_dIsVIPofDegreeD} also shows that, when restricted to $\FF\setminus\{0_\FF\}$, the map $p_d\circ\Phi_N$ is a VIP-system of degree $d$.
\end{rem}
\section{Preliminary stochastic constructions}
The following two results are needed for the proof of Theorem \hyperlink{0.thm:FailureOfPolynomialKhintchine}{A}. They generalize the Markov-process estimates used in the proof of \cite[Theorem A]{ZelIP0Khintchine2023}.
\begin{lem}\label[lemma]{3.lem:EstimateLemma}
    Let $d\in\N$ and let $(y_n)_{n\in\N}$ be a sequence of independent, uniformly distributed $\{0,...,d\}$-valued random variables. For each $N\in\N$ and each $j\in\{1,...,d\}$, set 
    $$
\Sigma_j^{(N)}=\sum_{n=1}^N\mathbbm 1_{\{j\}}(y_n).
    $$
Then, for every $\epsilon>0$ there is an $N_\epsilon\in\N$ such that for any $N\geq N_\epsilon$ and any $A\subseteq \{1,...,d\}$,
$$
|\mathbb P(\bigcap_{j=1}^{d}\{\Sigma_j^{(N)}\equiv \mathbbm 1_A(j)\mod 2\})-\frac{1}{2^{d}}|<\epsilon. 
$$
\end{lem}
\begin{proof}
    Let $A_1,...,A_{2^{d}}$ be an enumeration of the  subsets of  $\{1,...,d\}$ with $A_1=\emptyset$ and let $M\in \R^{2^{d}\times 2^{d}}$ be the matrix given by 
    $$
M_{i,j}=\begin{cases}
    \frac{1}{d+1},\text{ if }\exists t\in\{1,...,d\},\,A_i\triangle A_j=\{t\},\\
    \frac{1}{d+1},\text{ if }i=j,\\
    0,\, \text{ otherwise.}
\end{cases}
    $$
It is not hard to check that $M$ is a stochastic matrix and that every entry of the matrix $M^{2^d}$ is a positive real number. Noting that $(1,...,1)M=(1,...,1)$ and $M(1,...,1)^T=(1,...,1)^T$, by the  Perron-Frobenius theorem we obtain
$$
\lim_{N\rightarrow\infty}M^N=\begin{pmatrix}
    \frac{1}{2^d}&\cdots&\frac{1}{2^d}\\
    \vdots &\ddots& \vdots\\
    \frac{1}{2^d}&\cdots&\frac{1}{2^d}
\end{pmatrix}.
$$
Noting that 
$$
\lim_{N\rightarrow\infty} \begin{pmatrix}
    \mathbb P(\bigcap_{j=1}^{d}\{\Sigma_j^{(N)}\equiv \mathbbm 1_{A_1}(j)\mod 2\})\\
    \mathbb P(\bigcap_{j=1}^{d}\{\Sigma_j^{(N)}\equiv \mathbbm 1_{A_2}(j)\mod 2\})\\
    \vdots\\
    \mathbb P(\bigcap_{j=1}^{d}\{\Sigma_j^{(N)}\equiv \mathbbm 1_{A_{2^d}}(j)\mod 2\})
\end{pmatrix}=\lim_{N\rightarrow\infty}M^N\begin{pmatrix}
        1\\ 0 \\\vdots \\ 0
    \end{pmatrix}=\begin{pmatrix}
        \frac{1}{2^d}\\ \frac{1}{2^d} \\\vdots \\ \frac{1}{2^d}
    \end{pmatrix}
$$
we see that the result follows. 
\end{proof}
\begin{lem}\label{3.lem:MainEstimate}
    Fix $d\in\N$, $\epsilon>0$, and $k\in\N$.  Let $(y_n)_{n\in\N}$, $N:=N_{\epsilon}$,  and $\Sigma_j^{(kN)}$,  $j\in\{1,...,d\}$, be as in the statement of \Cref{3.lem:EstimateLemma}. Then, 
    $$
\left|\mathbb P(\sum_{t=1}^d\sum_{1\leq i_1<\cdots<i_t\leq d}\prod_{s=1}^t\Sigma_{i_s}^{(kN)}\equiv 0\mod 2)-\frac{1}{2^d}\right|<\epsilon.
    $$
\end{lem}
\begin{proof}
Observe that 
$$\sum_{t=1}^d\sum_{1\leq i_1<\cdots<i_t\leq d}\prod_{s=1}^t\Sigma_{i_s}^{(kN)}\equiv \prod_{t=1}^d(1+\Sigma_t^{(kN)})+1\mod 2.$$
Thus, one can find a $j\in\{1,...,d\}$ with $\Sigma_j^{(kN)}\equiv 1\mod 2$ if and only if 
    $$\sum_{t=1}^d\sum_{1\leq i_1<\cdots<i_t\leq d}\prod_{s=1}^t\Sigma_{i_s}^{(kN)}\equiv 1\mod 2.$$
    Invoking  \Cref{3.lem:EstimateLemma}, the result follows.
\end{proof}
\section{Proof of Theorem A}\label{4.Sec:MainResult}
 Fix $d\in\N$. We will first show that the conclusion of Theorem \hyperlink{0.thm:FailureOfPolynomialKhintchine}{A} holds for a non-ergodic system  $(Y,\mathcal B,\nu, (R^\xi)_{\xi\in\FF})$, a collection of sets $B_\delta$ with $\nu(B_\delta)=\delta$, $\delta\in (0,\frac{1}{2}]$, and the polynomial 
 $$
q_{d,\epsilon}:=q_d\circ \Phi_N,
 $$
 where $N\in\N$ is a natural number depending on each  $\epsilon\in (0,1/2)$ yet to be defined  and $q_d,\Phi_N$
 are as  in Section \ref{2.Sec}.
 Then we will explain how these objects can be employed to prove Theorem \hyperlink{0.thm:FailureOfPolynomialKhintchine}{A} in the correct generality.
\subsection{ Proof of a non-ergodic version of Theorem A}
Let $\lambda$ denote the Lebesgue measure on $[0,1)$ and let $\mathbb P$ denote the product probability measure on $\Omega:=\{0,...,d\}^\N$ induced by the normalized counting measure   on $\{0,...,d\}$. Set $Y=[0,1)\times\Omega$, 
let $\mathcal B$ denote the product $\sigma$-algebra on $Y$, let $\nu:=\lambda\otimes\mathbb P$, and for each $\delta\in (0,\frac{1}{2}]$, let $B_\delta:=[0,\delta)\times\Omega$. Clearly, $\nu(B_\delta)=\delta$ for each $\delta\in (0,\frac{1}{2}]$.\\
Identify $\FF$ with $\mathcal P_d$. We will let   $(R^\xi)_{\xi\in \FF}$ be the $\FF$-action generated by the family of $\nu$-preserving and commuting involutions
$$
R^{x_{i_1}^{(1)}\cdots x_{i_d}^{(d)}}:[0,1)\times\Omega\rightarrow [0,1)\times \Omega,(i_1,...,i_d)\in \N^d, 
$$
defined  as follows: If   there is a $t\in\{1,...,d\}$ with $i_1<\cdots<i_t=i_{t+1}=\cdots=i_d$, we set 
\begin{equation}\label{4.defnInfolutionRCase1}
R^{x_{i_1}^{(1)}\cdots x_{i_d}^{(d)}}(x,(y_n)_{n\in\N}):=(x+\frac{1}{2}\sum_{\substack{\alpha\subseteq\{1,...,d\},\\|\alpha|=t}}\sum_{\substack{\sigma:\alpha\rightarrow\{i_1,...,i_t\}\\\text{ is a bijection}}}\left(\prod_{j\in\alpha}\mathbbm 1_{\{j\}}(y_{\sigma(j)})\right)\mod 1,(y_n)_{n\in\N}). 
\end{equation}
In all other cases, we set  
\begin{equation}\label{4.defnInfolutionRCase2}
R^{x_{i_1}^{(1)}\cdots x_{i_d}^{(d)}}(x,(y_n)_{n\in\N})=(x,(y_n)_{n\in\N}).
\end{equation}
(Notice that because every map of the form $R^{x_{i_1}^{(1)}\cdots x_{i_d}^{(d)}}$ only affects the first coordinate by adding either $0$ or $1/2$ and leaves the second coordinate fixed, these maps are commuting $\nu$-preserving involutions.) \\ 
Pick $\epsilon\in (0,1/2)$ and set $N:=N_{\epsilon}$, where $N_\epsilon$ is as in \Cref{3.lem:EstimateLemma}, and let $\Phi_N:\FF\rightarrow \FF$ be as defined in Section \ref{2.Sec} (i.e. it is given by $\Phi_N(\xi)=\langle \alpha_\xi\rangle_N$). In other words, after identifying $\FF$ with $\mathcal F_\emptyset$, 
$$
\Phi_{N}(\xi)=\bigcup_{j\in\alpha_\xi}\{N(j-1)+1,...,Nj\}.
$$
Without loss of generality, assume $N>d$. Clearly, $\Phi_{N}$ is an injective homomorphism. Take now $\delta\in (0,1/2]$ and $\xi\in \FF$, $\xi\neq 0_\FF$. Set $\gamma=\langle \alpha_\xi\rangle_{N}$ and let $q_d$ be as in \eqref{2.eq:Defnq_d}. Note that 
\begingroup
\allowdisplaybreaks
\begin{multline}
    \nu(B_\delta\cap R^{-q_d(\Phi_N(\xi))} B_\delta)=(\lambda\otimes\mathbb P)(B_\delta\cap R^{-\prod_{t=1}^d(\sum_{i\in\gamma} x_i^{(t)})}B_\delta)
    =\int_{[0,1)}\int_\Omega \mathbbm 1_{[0,\delta)}(x)\mathbbm 1_{[0,\delta)}\\
    \left(x+\frac{1}{2}\sum_{t=1}^d
    \sum_{\beta\subseteq \gamma,\,|\beta|=t}
    \sum_{\substack{\alpha\subseteq\{1,...,d\},\\|\alpha|=t}}\sum_{\substack{\sigma:\alpha\rightarrow\beta\\\text{ is a bijection}}}\prod_{j\in\alpha}\mathbbm 1_{\{j\}}(\omega_{\sigma(j)})\mod 1\right)\text{d}\mathbb P(\omega)\text{d}\lambda(x)\\
    =\int_{[0,1)}\int_\Omega \mathbbm 1_{[0,\delta)}(x)\mathbbm 1_{[0,\delta)}(x+\frac{1}{2}\sum_{t=1}^d\sum_{1\leq i_1<\cdots<i_t\leq d}\prod_{s=1}^t\Sigma_{i_s}^{(|\alpha_\xi|N)}(\omega)\mod 1)\text{d}\mathbb P(\omega)\text{d}\lambda(x)\\
    =\int_{[0,1)} \mathbbm 1_{[0,\delta)}(x)\mathbb P\left(\sum_{t=1}^d\sum_{1\leq i_1<\cdots<i_t\leq d}\prod_{s=1}^t\Sigma_{i_s}^{(|\alpha_\xi|N)}\equiv0\mod 2\right)\text{d}\lambda(x).
\end{multline}
\endgroup
So, by Lemma \ref{3.lem:MainEstimate}, we obtain
\begingroup
\allowdisplaybreaks
\begin{multline}\label{4.eq:mainEstimate}
|\nu(B_\delta\cap R^{-q_d(\Phi_N(\xi))}B_\delta)-\frac{\delta}{2^d}|\\
=|\int_{[0,1)} \mathbbm 1_{[0,\delta)}(x)\mathbb P\left(\sum_{t=1}^d\sum_{1\leq i_1<\cdots<i_t\leq d}\prod_{s=1}^t\Sigma_{i_s}^{(|\alpha_\xi|N)}\equiv0\mod 2\right)\text{d}\lambda(x)-\frac{\delta}{2^d}|\\
=\delta|\mathbb P\left(\sum_{t=1}^d\sum_{1\leq i_1<\cdots<i_t\leq d}\prod_{s=1}^t\Sigma_{i_s}^{(|\alpha_\xi|N)}\equiv0\mod 2\right)-\frac{1}{2^d}|
<\delta\cdot\epsilon.
\end{multline}
\endgroup
\subsection{Proof of Theorem A in full generality}
We are now in position to construct the weakly mixing system $(X,\mathcal A,\mu, (T^\xi)_{\xi\in\FF})$  and the family of sets $A_\delta$, $\delta\in (0,1/2]$, with $\mu(A_\delta)=\delta$ for which \eqref{0.eq:SmallIntersections} holds.\\
We will let $X=\prod_{\Gamma\in\FF}Y$, $\mathcal A$ denote the product $\sigma$-algebra on $X$, and set $\mu$ to be the product measure induced by $\nu$ on $X$. For each $\xi\in\FF$, define the projection $\pi_\xi:X\rightarrow Y$ by $\pi_\xi(\omega)=\omega(\xi)$. For each  $\delta\in (0,1/2]$, we set $A_\delta=\pi^{-1}_{0_\FF}(B_\delta)$.\\
Identify $\FF$ with $\FF\times\FF$. The $\mu$-preserving action $(T^{(\xi,\xi')})_{(\xi,\xi')\in\FF\times\FF}$ is the  $\FF\times\FF$-action generated by the $\mu$-preserving and commuting involutions of the form 
$$
T^{(\xi,0_\FF)}:X\rightarrow X\text{ and }T^{(0_\FF,\xi')}:X\rightarrow X,\,\xi,\xi'\in\FF,
$$
defined by 
$$
[T^{(\xi,0_\FF)}\omega](\eta)=R^\xi(\omega(\eta))
$$
and
$$
[T^{(0_\FF,\xi')}\omega](\eta)=\omega(\eta+\xi').
$$
(Note that we indeed have that for any $\xi,\xi'\in \FF$, $T^{(\xi,0_\FF)}$ and $T^{(0_\FF,\xi')}$ are involutions and that  $T^{(\xi,0_\FF)}T^{(0_\FF,\xi')}=T^{(0_\FF,\xi')}T^{(\xi,0_\FF)}$).\\
To check that $(T^{(\xi,\xi')})_{(\xi,\xi')\in\FF\times \FF}$ is weakly mixing, let $\xi'_k\in \FF$, $k\in\N$, be such that $\alpha_{\xi'_k}=\{k\}$. It follows that for any $A,B\in\mathcal A$, 
$$
\lim_{k\rightarrow\infty}\mu(A\cap T^{(0_\FF,\xi'_k)}B)=\mu(A)\mu(B).
$$
Recall that the polynomial $p_d:\FF\rightarrow \FF$ 
(where we view the codomain of $p_d$ as $\FF\times \FF$)  is given by the rule $p_d(\xi)=(q_d(\xi),0_{\FF})$. Thus, for any $\epsilon\in (0,1/2)$ and $N$ picked as before, we have that for any $\delta\in (0,1/2]$ and any $\xi\in\FF$,
$$
\mu(A_\delta \cap T^{-p_d(\Phi_{N}(\xi))}A_\delta)=\mu(\pi^{-1}_{0_\FF}(B_\delta\cap R^{-q_d(\Phi_{N}(\xi))}B_\delta))=\nu(B_\delta\cap R^{-q_d(\Phi_{N}(\xi))}B_\delta),
$$
and, so,  
$$\{\xi\in \mathbb F_2^\omega\,|\,|\mu(A_\delta\cap T^{-(p_d\circ\Phi_{N})(\xi)}A_\delta)-\frac{\delta}{2^{d}}|\geq \epsilon\cdot\delta\}=\{0_\FF\}.$$
That  $p_d\circ \Phi_N$ is a polynomial  of degree $d$ follows from  Lemma \ref{2.p_dIsVIPofDegreeD}. 
\qedsymbol 
\subsection{One further consequence of Theorem A}
A closer look at the proof of Theorem \hyperlink{0.thm:FailureOfPolynomialKhintchine}{A} yields the following result.
\begin{cor}
    Fix $d\in\N$ and let $p_d$, $(X,\mathcal A,\mu,(T^\xi)_{\xi\in\FF})$, and $A_\delta$, $\delta \in (0,1/2]$, be as in the statement of Theorem \hyperlink{0.thm:FailureOfPolynomialKhintchine}{A}. There is a degree $d$ polynomial $p:\FF\rightarrow \FF$ vanishing at zero such that for every $\delta\in (0,1/2]$, the  only accumulation point of the set
    $$
\{\mu(A_\delta\cap T^{-p(\xi)}A_\delta)\,|\,\xi\in\FF\}
    $$
 is $\delta/2^d$. In other words, for every $\epsilon>0$, the set 
    \begin{equation}\label{4.eq:FiniteSet}
\{\xi\in\FF\,|\,|\mu(A_\delta\cap T^{-p(\xi)}A_\delta)-\frac{\delta}{2^d}|\geq \epsilon\}
    \end{equation}
    is finite. (As we will see, one can take $p=p_d\circ \varphi$, where $\varphi$ is an injective homomorphism.) 
\end{cor}
\begin{proof}
    Define the injective homomorphism $\varphi:\FF\rightarrow\FF$ by 
    $$
\varphi(\alpha)=\alpha\cap\{1,...,d\}\cup\bigcup_{j\in\alpha,\,j>d}\{d+1+(j-d-1)^2,...,d+(j-d)^2\}.
    $$
    Set $p=p_d\circ \varphi$. Since $(p_d\circ \varphi)(\alpha)=p_d(\alpha)$ for each $\alpha\subseteq\{1,...,d+1\}$, it follows that $p_d\circ \varphi$ is a polynomial of degree $d$. That \eqref{4.eq:FiniteSet} holds for every $\epsilon>0$ follows from the fact that $\lim_{\max \alpha\rightarrow\infty}|\varphi(\alpha)|=\infty$.
\end{proof}

\section{Theorem A is "sharp".}\label{5.Sec}
In this Section we prove the following extension of  Corollary \hyperlink{0.prop:CorrectRateOfGrowth}{C}.
For each $d\in\N$, we let $\alpha_d$ be the largest positive number with the property that for every invertible probability preserving system $(X,\mathcal A,\mu, (T^\xi)_{\xi\in\FF})$, every set $A\in\mathcal A$ with $\mu(A)=\frac{1}{2}$, and every VIP-system  $v:\FF\rightarrow \FF$ of degree $d$, one has
$$
\sup_{\xi\in\FF,\,\xi\neq0_\FF}\mu(A\cap T^{-v(\xi)}A)\geq \alpha_d.
$$
\begin{cor}
    For every $d\in\N$, 
    \begin{equation}\label{6.eq:SharpnessPoly}
    \frac{1}{2^{d+1}}\left(\frac{1}{2-1/2^d}\right)\leq \alpha_d\leq \frac{1}{2^{d+1}}.
    \end{equation}
    Thus, $(\alpha_d)_{d\in\N}$ is $O(\frac{1}{2^{d+1}})$.
\end{cor}

\begin{proof}
    Fix $d\in\N$. By Theorem \hyperlink{0.thm:FailureOfPolynomialKhintchine}{A} (and Remark \ref{2.rem:VipDegOFp_d}), we have that $\alpha_d\leq \frac{1}{2^{d+1}}$. Thus, all that remains to be shown is that $\frac{1}{2^{d+1}}\left( \frac{1}{2-1/2^d} \right)\leq \alpha_d$.\\
    To do this, let $\Gamma$ denote the Pontryagin dual of $\FF$, let $(X,\mathcal A,\mu,(T^\xi)_{\xi\in\FF})$ be  a probability preserving system, let $A\in\mathcal A$ be such that $\mu(A)=\frac{1}{2}$, and let $v:\FF\rightarrow \FF$ be a VIP-system of degree $d$. Let 
    $$f_A:=\frac{\mathbbm 1_A-\mu(A)}{\sqrt{\mu(A)}\sqrt{1-\mu(A)}}=2(\mathbbm 1_A-\mu(A)).$$
Note that $\| f_A\|_{L^2(\mu)}=1$. Thus, by Bochner's theorem, there is a unique Borel probability measure $\sigma$ on $\Gamma$ with the property that for each $\xi\in\FF$,
$$\langle T^\xi f_A,f_A\rangle =\int_\Gamma \chi(\xi)\text{d}\sigma(\chi).$$
Consider now any $\xi_0,...,\xi_d\in\FF\setminus\{0_\FF\}$  with the property that the sets $\alpha_{\xi_i}$, $i\in\{0,...,d\}$, are pairwise disjoint.
By the definition of VIP-systems of degree $d$, we have 
\begin{equation}\label{5.VIPIdentity}
0_\FF=\sum_{t=1}^{d+1}\left((-1)^t\sum_{0\leq i_1<\cdots<i_t\leq d}v(\sum_{j=1}^t \xi_{i_j})\right)=\sum_{t=1}^{d+1}\sum_{0\leq i_1<\cdots<i_t\leq d}v(\sum_{j=1}^t \xi_{i_j}).
\end{equation}
Let $\chi\in\Gamma$ and note that for every $\xi\in\FF$, $\chi(\xi)\in\{-1,1\}$. Since  formula \eqref{5.VIPIdentity} implies that 
$$
1=\chi(0_\FF)=\prod_{t=1}^{d+1}\prod_{0\leq i_1<\cdots<i_t\leq d}\chi\left(v(\sum_{j=1}^t \xi_{i_j})\right),
$$
and \eqref{5.VIPIdentity} has an odd number of summands, we obtain that 
$$\chi(v(\sum_{j\in\alpha}\xi_{j}))=1$$
for some non-empty $\alpha\subseteq \{0,...,d\}$. Thus, we can find a non-empty $\alpha\subseteq \{0,...,d\}$ for which 
$$\sigma(\{\chi\in\Gamma\,|\,\chi\left(v(\sum_{j\in\alpha}\xi_j)\right)=1\})\geq \frac{1}{2^{d+1}-1}.
$$
It follows that 
\begingroup
\allowdisplaybreaks
\begin{multline*}
\frac{1}{2^{d+1}}\frac{1}{2-1/2^{d}}=\frac{1}{4}\left(1-\frac{2^{d+1}-3}{2^{d+1}-1} \right)=\frac{1}{4}-\frac{1}{4}\left(\frac{2^{d+1}-2}{2^{d+1}-1}-\frac{1}{2^{d+1}-1} \right)\\
\leq \mu^2(A)-\frac{1}{4}\left(\sigma(\{\chi\in\Gamma\,|\,\chi(v\left(\sum_{j\in\alpha}\xi_j\right))=-1\})- \sigma(\{\chi\in\Gamma\,|\,\chi(v\left(\sum_{j\in\alpha}\xi_j\right))=1\}) \right)\\
=\mu^2(A)+\frac{1}{4}\int_\Gamma \chi\left(v(\sum_{j\in\alpha}\xi_j)\right)\text{d}\sigma(\chi)=\mu(A\cap T^{-v(\sum_{j\in\alpha}\xi_j)}A).
\end{multline*}
\endgroup
We are done. 
\end{proof}
\section{ The proof of Corollary D}\label{6.Sec}
Our proof of Corollary \hyperlink{0.cor:CombinatorialApplication2}{D} makes use of the 
following inverse form  of Furstenberg's correspondence principle.
    \begin{thm}[Theorem 9.1 in \cite{BerZel-NiceRecurrence}]\label{6.Lem:epsilonInverse} 
    Let $(G,+)$ be a countable abelian group, let $(X,\mathcal A,\mu)$  be a standard Lebesgue space with no atoms (i.e. it is measure-theoretically isomorphic to $[0,1]$ equipped with the Lebesgue measure), and let $(S^g)_{g\in G}$ denote a $\mu$-preserving $G$-action. Suppose that $(S^g)_{g\in G}$ is ergodic. For any $A\in\mathcal A$ with $\mu(A)>0$  and any $\delta>0$, there is a set $E\subseteq G$ such that for any $k\in\N$,  $g_1,...,g_k\in G$, there is a non-negative real number $r_{E,g_1,...,g_k}$ with the properties that (a) for any F{\o}lner sequence $(\Phi_N)_{N\in\N}$ in $G$,
    $$r_{E,g_1,...,g_k}=\lim_{N\rightarrow\infty}\frac{|\bigcap_{j=1}^k(E-g_j)\cap \Phi_N|}{|\Phi_N|}$$
    and (b) $\left|r_{E,g_1,...,g_k}-\mu(\bigcap_{j=1}^k S^{-g_j}A)\right|<k\delta$.
\end{thm}
\begin{proof}[Proof of Corollary \hyperlink{0.cor:CombinatorialApplication2}{D}]
Fix $d\geq 2$, $\epsilon=(1/64)^2$ and let $(X,\mathcal A,\mu, (T^{(\xi,\xi')})_{(\xi,\xi')\in\FF\times\FF})$, $N:=N_\epsilon\in\N$, and $A_{1/2}$ be as in the proof of Theorem \hyperlink{0.thm:FailureOfPolynomialKhintchine}{A}. Identifying $\FF\times\FF$ with $\mathcal P_d\times \FF$, we define for every $i_1,...,i_d\in\N$ and every $\xi'\in\FF$, the $\mu$-preserving involution $S^{(x_{i_1}^{(1)}\cdots x_{i_d}^{(d)},\xi')}$ by
$$
S^{(x_{i_1}^{(1)}\cdots x_{i_d}^{(d)},\xi')}:=T^{(0,\xi')}\prod_{s_1,...,s_d=0}^{N-1}T^{(x_{(Ni_1-s_1)}^{(1)}\cdots x_{(Ni_d-s_d)}^{(d)},0)}.
$$
Since the transformations on the right-hand side commute and are involutions, one obtains that $(S^{(\xi,\xi')})_{(\xi,\xi')\in\mathcal P_d\times\FF}$ is a $\mu$-preserving $\mathcal P_d\times\FF$-action.\\
It follows that $(X,\mathcal A,\mu, (S^{(\xi,\xi')})_{(\xi,\xi')\in\mathcal P_d\times\FF})$ is a weakly mixing, invertible probability preserving system and that for any $\xi\in\FF$,
$$
\mu(A_{1/2}\cap T^{-p_d(\Phi_N(\xi))}A_{1/2})=\mu(A_{1/2}\cap S^{-p_d(\xi)}A_{1/2}), 
$$
where, as before, we view $p_d$ as a map from $\FF$ to $\mathcal P_d\times\FF$ given by $p_d(\xi)=(q_d(\xi),0_\FF)$. Indeed, the definition of $(S^{(\xi,\xi')})_{(\xi,\xi')\in\mathcal P_d\times\FF}$ was chosen precisely so that $S^{(q_d(\xi),0_\FF)}=T^{(q_d(\Phi_N(\xi)),0_\FF)}$. That 
$(S^{(\xi,\xi')})_{(\xi,\xi')\in\mathcal P_d\times\FF}$ is weakly mixing follows from the equality $T^{(0_{\mathcal P_d},\xi')}=S^{(0_{\mathcal P_d},\xi')}$.\\ 
Thus, by our choice of $N$ and $\epsilon$, we have that for every $\xi\neq 0_\FF$, 
\begin{equation}\label{6.eq:KeyInequality}
\mu(A_{1/2}\cap S^{-p_d(\xi)}A_{1/2})\leq \frac{1}{8}+\epsilon< \frac{1}{8}+\frac{1}{32}=\frac{5}{32}.
\end{equation}
Note that, by construction, $X$ is measure-theoretically isomorphic to a compact metric space, $\mathcal A=\text{Borel}(X)$, and $\mu$ has no atoms. Thus, $(X,\mathcal A,\mu)$ is an atomless, standard Lebesgue space (see \cite[Theorem 2.1]{waltersIntroduction}, for example). It follows that  there is a set $E\subseteq \mathcal P_d\times\FF$ for which properties (a) and (b) in  Theorem \ref{6.Lem:epsilonInverse} 
hold with $A=A_{1/2}$ and $\delta=\min\{1/64,1/2^{d+3}\}$. Identifying $\mathcal P_d\times \FF$ with $\FF$, we can assume without loss of generality that $E\subseteq \FF$.
By (b) and our choice of $\delta$ it follows that 
$$
d^*(E)> \mu(A_{1/2})-\delta\geq \max\{1/2-1/64, 1/2-1/2^{d+3}\}\geq\max\{1/2-1/64, 1/2-1/32\}.
$$
Also, by (b) and \eqref{6.eq:KeyInequality}, for every non-zero $\xi\in\FF$,
$$
d^*(E\cap (E-p_d(\xi)))<\mu(A_{1/2}\cap S^{-p_d(\xi)}A_{1/2})+2\delta\leq \frac{5}{32}+\frac{1}{32}=\frac{6}{32}<\frac{1}{4}-\frac{1}{32}.
$$
Thus, for every non-zero $\xi\in\FF$,
$$
d^*(E\cap (E-p_d(\xi)))<\frac{1}{4}-\frac{1}{32}<\left(\frac{1}{2}-\frac{1}{32}\right)^2\leq (d^*(E))^2-\epsilon.
$$
We are done. 
\end{proof}
It is worth noting that the argument used here to prove Corollary \hyperlink{0.cor:CombinatorialApplication2}{D} reveals an alternative form of  Theorem \hyperlink{0.thm:FailureOfPolynomialKhintchine}{A} in which instead of varying $p_{d,\epsilon}$ depending on $\epsilon$ while keeping the probability preserving system $(X,\mathcal A,\mu,(T^\xi)_{\xi\in\FF})$ fixed, one keeps $p_d$ fixed while varying $(T^\xi)_{\xi\in\FF}$ depending on $\epsilon$. 
\begin{namedthm*}{Theorem A$^\prime$}
Fix $d\in\N$ and let $p_d$, $(X,\mathcal A,\mu)$, and $A_\delta\in\mathcal A$, $\delta\in (0,1/2]$, be as in the proof of Theorem \hyperlink{0.thm:FailureOfPolynomialKhintchine}{A}. For every $\epsilon\in(0,1/2)$, there is an invertible, weakly mixing $\mu$-preserving $\FF$-action $(T^\xi)_{\xi\in\FF}$ with the property   that 
\begin{equation}\label{6.eq:SmallIntersections}
\{\xi\in \mathbb F_2^\omega\,|\,|\mu(A_\delta\cap T^{-p_{d}(\xi)}A_\delta)-\frac{\delta}{2^{d}}|\geq \epsilon\cdot\delta\}=\{0_\FF\}
\end{equation} 
for each $\delta\in (0,1/2]$. 
\end{namedthm*}
\appendix

\counterwithin{thm}{section}
\renewcommand{\thethm}{\Alph{section}\arabic{thm}}
\renewcommand{\thesubsection}{\Alph{section}.\arabic{subsection}}
\renewcommand{\thesection}{\Alph{section}}
\section{Proof of a (highly) restricted case of the density polynomial Hales-Jewett conjecture}\label{A.Sec}
In this appendix we prove the following special case of the density polynomial Hales-Jewett conjecture. We formulate this special case in a way which will facilitate the proof. \\
For any $d,N\in\N$, let $[N]:=\{1,...,N\}$ and let  $F_N^{(d)}:=\{-1,1\}^{[N]^d}$. For any $\Gamma\subseteq [N]^d$, set 
$$F_N^{(d)}(\Gamma):=\{f\in F_N^{(d)}\,|\,\forall x\in [N]^d\setminus \Gamma,\,f(x)=1\}.$$
Note that for any $f,g\in F_N^{(d)}(\Gamma)$, the pointwise product of $f$ and $g$, which we denote by $fg$, also belongs to 
$F_N^{(d)}(\Gamma)$ and, so, $F_N^{(d)}(\Gamma)$ is a subgroup of the multiplicative abelian group $F_N^{(d)}$.  Note that for any $f\in F_N^{(d)}$ there is a unique $A\subseteq [N]^d$ such that 
$$f=1-2\mathbbm 1_A.$$
\begin{thm}[Cf. Theorem 1.3 in \cite{MR5029922}]\label{A.thm:EpsilonPDHJ}
    For every $d\in\N$ there is an $\epsilon>0$  and an $N:=N_d\in\N$ with the property that for any $\mathcal S\subseteq F_N^{(d)}$ with $|\mathcal S|>(\frac{1}{2}-\epsilon)|F_N^{(d)}|$ one can find $f,g\in\mathcal S$ and a non-empty $\alpha\subseteq [N]$ such that 
    $$f|_{\alpha^d}= 1\text{ and }(1-2\mathbbm 1_{\alpha^d})f=g.$$ 
\end{thm}
The proof of \cref{A.thm:EpsilonPDHJ} will make use of the following two lemmas and some additional notation.
Fix $d\in\N$ and let $D=d+1$. For $N\in\N$ and  $t\in\N\cap(0,N/D]$, we define the set 
$$A_t:=\{(t-1)D+1,...,tD\}^d$$
and set $A_0=\emptyset$. For any $\mathcal Q\subseteq F_N^{(d)}$, we define the function $f_\mathcal Q:F_N^{(d)}\rightarrow \{-1,1\}$ by 
$$f_\mathcal Q(x)=1-2\mathbbm 1_\mathcal Q(x).$$
It follows that for any $\mathcal P,\mathcal Q,\mathcal R\subseteq F_N^{(d)}$ with $\mathcal P\neq\emptyset$,  
$$
\sum_{x\in \mathcal P}|f_\mathcal Q(x)-f_\mathcal R(x)|=2|(\mathcal Q\cap \mathcal P)\triangle(\mathcal R\cap \mathcal P)|$$
 and
 $$
 \sum_{x\in \mathcal P}|f_\mathcal Q(x)+f_\mathcal R(x)|=2(|(\mathcal Q\cap \mathcal P)\cap(\mathcal R\cap \mathcal P)|+|(\mathcal Q^c\cap \mathcal P)\cap(\mathcal R^c\cap \mathcal P)|).
 $$
For any $\mathcal Q\subseteq F_N^{(d)}$ and any $A\subseteq [N]^d$, we will let 
$$\sigma_A (f_\mathcal Q)=f_{(1-2\mathbbm 1_A)\cdot\mathcal Q}.$$
Let $(c_k)_{k=0}^\infty$ denote the sequence defined recursively by the rules (1) $c_0=2$ and (2) for every $k\in\N$, 
\begin{equation}\label{eq:Defn(c_k)}
c_k=\sum_{r=0}^{k-1}\binom{k}{r}c_r.
\end{equation}
\begin{lem}\label{A.lem:HomogeneousCondition}
Let $\epsilon=\frac{1}{4(c_D+1)}$, let 
$\tau\in\N$ be  the least natural number $n$ such that  
\begin{equation}\label{eq:DefnTau}
   (1+\frac{1}{2^{D^d}}\frac{2\epsilon}{1-2\epsilon})^n >\frac{2}{1-2\epsilon},
\end{equation}
and let $N:=\tau D$. For any $\mathcal S\subseteq F_N^{(d)}$ with $|\mathcal S|>(\frac{1}{2}-\epsilon)|F_N^{(d)}|$, there is a $t\in\{1,...,\tau\}$ with the property that for some $\xi\in F_N^{(d)}(\bigcup_{j=0}^{t-1}A_j)$ and  every $\gamma\in F_N^{(d)}(A_t)$,
\begin{equation}\label{A.eq:HomogeneousCondition}
    |(\gamma\xi\cdot\mathcal S)\cap F_N^{(d)}(\left(\bigcup_{j=0}^t A_j\right)^c)|>(\frac{1}{2}-2\epsilon)|F_N^{(d)}(\left(\bigcup_{j=0}^t A_j\right)^c)|.
\end{equation}
\end{lem}
\begin{lem}\label{A.lem:RecurrenceLem}
    Let $t,N\in\N$ be such that $Dt\leq N$, set $X=F_N^{(d)}(\left(\bigcup_{j=0}^tA_j\right)^c)$, and let $\epsilon=\frac{1}{4(c_D+1)}$. 
    For any $\mathcal Q\subseteq F_N^{(d)}$
there is a non-empty $\alpha\subseteq \{(t-1)D+1,...,tD\}$ and a $\xi\in F_N^{(d)}(A_t\setminus \alpha^d)$ such that 
\begin{equation}\label{A.eq:SmallSymmetricDiff}
\sum_{x\in X}|f_{\xi\cdot \mathcal Q}(x)-\sigma_{\alpha^d}f_{\xi\cdot\mathcal Q}(x)|\leq 2(1-4\epsilon)|X|.
\end{equation}
\end{lem}
\begin{proof}[ Proof of \cref{A.thm:EpsilonPDHJ}]
    Fix $d\in\N$, let $\epsilon=\frac{1}{4(c_D+1)}$, and let $\tau,N$ be as in the statement of Lemma \ref{A.lem:HomogeneousCondition}. Let $\mathcal S\subseteq F_N^{(d)}$ be such that $|\mathcal S|>(\frac{1}{2}-\epsilon)|F_N^{(d)}|$. By  Lemma \ref{A.lem:HomogeneousCondition}, we can find a $t\in\{1,...,\tau\}$ and a
 $\xi_1\in F_N^{(d)}(\bigcup_{j=0}^{t-1}A_j)$ with the property that for   every $\gamma\in F_N^{(d)}(A_t)$, formula \eqref{A.eq:HomogeneousCondition} holds. Set $\mathcal Q=(\xi_1\cdot\mathcal S)\cap F_N^{(d)}(\left(\bigcup_{j=0}^{t-1}A_j\right)^c) $ and let 
 $$
X=F_N^{(d)}(\left(\bigcup_{j=0}^tA_j\right)^c).
 $$
 By Lemma \ref{A.lem:RecurrenceLem}, there is a non-empty $\alpha\subseteq \{(t-1)D+1,...,tD\}$ and a $\xi_2\in F_N^{(d)}(A_t\setminus \alpha^d)$ such that 
 $$
| \Big[(\xi_2\cdot\mathcal Q)\triangle 
 ((1-2\mathbbm 1_{\alpha^d})\xi_2\cdot\mathcal Q)\Big]\cap X|\leq (1-4\epsilon)|X|.
 $$
 Because Lemma \ref{A.lem:HomogeneousCondition} implies that 
$$
|(\xi_2\xi_1\cdot\mathcal S)\cap X|,|((1-2\mathbbm 1_{\alpha^d})\xi_2\xi_1\cdot\mathcal S)\cap X|>(\frac{1}{2}-2\epsilon)|X|,
$$
we can find an $f\in (\xi_2\xi_1\cdot\mathcal S)\cap X$ such that $(1-2\mathbbm 1_{\alpha^d})f\in  (\xi_2\xi_1\cdot\mathcal S)$. It follows that 
$\xi_1f\in (\xi_2\cdot\mathcal S)\cap F_N^{(d)}(A_t^c)$
and, so, that 
$$
\xi_2\xi_1f\in \mathcal S\cap F_N^{(d)}((\alpha^d)^c)\text{ and }(1-2\mathbbm 1_{\alpha^d})\xi_2\xi_1f\in\mathcal S. 
$$
We are done. 
\end{proof}
\subsection{The proof of Lemma \ref{A.lem:HomogeneousCondition}}
Set $\delta=\frac{1}{2^{D^d}}\frac{2\epsilon}{1-2\epsilon}$. We claim that in order to prove \eqref{A.eq:HomogeneousCondition}, it suffices to show that for some $t\in\{1,...,\tau\}$ and some $\xi\in F_N^{(d)}(\bigcup_{j=0}^{t-1} A_j)$, one has that  for every $\gamma\in F_N^{(d)}(A_t)$,
\begin{equation}\label{A.eq:UniformUpperBound}
    |(\gamma\xi\cdot\mathcal S)\cap F_N^{(d)}(\left(\bigcup_{j=0}^t A_j\right)^c)|<(1+\delta)^t(\frac{1}{2}-\epsilon)|F_N^{(d)}(\left(\bigcup_{j=0}^t A_j\right)^c)|
\end{equation}
and that 
\begin{equation}\label{A.eq:LowerBoundForGoodCilinder}
    (1+\delta)^{t-1}(\frac{1}{2}-\epsilon)|F_N^{(d)}(\left(\bigcup_{j=0}^{t-1}A_j\right)^c)|\leq |(\xi\cdot \mathcal S)\cap F_N^{(d)}(\left(\bigcup_{j=0}^{t-1}A_j\right)^c)|
\end{equation}
Indeed, fix $\gamma\in F_N^{(d)}(A_t)$. Note that 
\begingroup
\allowdisplaybreaks
\begin{multline}
|(\xi\cdot \mathcal S)\cap F_N^{(d)}(\left(\bigcup_{j=0}^{t-1}A_j\right)^c)|\\
=\sum_{\eta\in F_N^{(d)}(A_t)}|(\xi\cdot \mathcal S)\cap[\eta\cdot F_N^{(d)}(\left(\bigcup_{j=0}^{t}A_j\right)^c)]|
=\sum_{\eta\in F_N^{(d)}(A_t)}|(\eta\xi\cdot \mathcal S)\cap F_N^{(d)}(\left(\bigcup_{j=0}^{t}A_j\right)^c)|\\
<|(\gamma\xi\cdot\mathcal S)\cap F_N^{(d)}(\left(\bigcup_{j=0}^t A_j\right)^c)|+(2^{D^d}-1)(1+\delta)^t(\frac{1}{2}-\epsilon)|F_N^{(d)}(\left(\bigcup_{j=0}^t A_j\right)^c)|.
\end{multline}
\endgroup
So, because $1+(1-2^{D^d})\delta>0$,
\begingroup
\allowdisplaybreaks
\begin{multline}
(\frac{1}{2}-\epsilon)(1+(1-2^{D^d})\delta)|F_N^{(d)}(\left(\bigcup_{j=0}^{t}A_j\right)^c)|\\
=(\frac{1}{2}-\epsilon)|F_N^{(d)}(\left(\bigcup_{j=0}^{t-1}A_j\right)^c)|-(2^{D^d}-1)(1+\delta)(\frac{1}{2}-\epsilon)|F_N^{(d)}(\left(\bigcup_{j=0}^t A_j\right)^c)|\\
   \leq  (1+\delta)^{t-1}(\frac{1}{2}-\epsilon)|F_N^{(d)}(\left(\bigcup_{j=0}^{t-1}A_j\right)^c)|-(2^{D^d}-1)(1+\delta)^t(\frac{1}{2}-\epsilon)|F_N^{(d)}(\left(\bigcup_{j=0}^t A_j\right)^c)|\\
   <|(\gamma\xi\cdot\mathcal S)\cap F_N^{(d)}(\left(\bigcup_{j=0}^t A_j\right)^c)|.
\end{multline}
\endgroup
The result now follows by noting that 
\begin{multline*}
(\frac{1}{2}-\epsilon)\left(1+(1-2^{D^d})\delta\right)
    = (\frac{1}{2}-\epsilon)\left( 1+\left(\frac{1}{2^{D^d}}-1\right)\frac{2\epsilon}{1-2\epsilon}\right)\\
    =\frac{1-2\epsilon}{2}(\frac{1-4\epsilon}{1-2\epsilon}+\frac{1}{2^{D^d}}\frac{2\epsilon}{1-2\epsilon})>\frac{1}{2}-2\epsilon.
\end{multline*}
We now prove that \eqref{A.eq:UniformUpperBound} and \eqref{A.eq:LowerBoundForGoodCilinder} hold. To do this, suppose for contradiction that there is no $t\in\{1,...,\tau\}$ and no $\xi\in F_N^{(d)}(\bigcup_{j=0}^{t-1}A_j)$ for which both \eqref{A.eq:UniformUpperBound} and \eqref{A.eq:LowerBoundForGoodCilinder} hold simultaneously. We claim that there are functions $\gamma_0,\gamma_1,...,\gamma_\tau\in F_N^{(d)}$ with $\gamma_j\in F_N^{(d)}(A_j)$ for each $j\in\{0,...,\tau\}$ (so, in particular, $\gamma_0=\mathbbm 1_{[N]^d}$) and the property that for each $r\in\{0,...,\tau\}$,
\begin{equation}\label{A.eq:TooBig}
|\left(\prod_{j=0}^r \gamma_j\cdot \mathcal S\right)\cap F_N^{(d)}(\left(\bigcup_{j=0}^r A_j\right)^c)|\geq (1+\delta)^r(\frac{1}{2}-\epsilon)|F_N^{(d)}(\left(\bigcup_{j=0}^r A_j\right)^c)|
\end{equation}
Indeed, when $r=0$, \eqref{A.eq:TooBig} becomes the assumption of the Lemma. Furthermore, if \eqref{A.eq:TooBig} holds for some $r\in\{0,...,\tau-1\}$, then \eqref{A.eq:LowerBoundForGoodCilinder} holds with $r=t-1$ and
$\xi=\prod_{j=0}^r\gamma_j$. So, by our assumption, we can find $\gamma_{r+1}\in F_N^{(d)}(A_{r+1})$ for which  \eqref{A.eq:UniformUpperBound} does not hold. It follows that the functions $\gamma_1,...,\gamma_{r+1}$ satisfy \eqref{A.eq:TooBig}. Starting  with $r=0$, we can now apply this algorithm to find $\gamma_1,...,\gamma_\tau$ for which \eqref{A.eq:TooBig} holds.
Taking $r=\tau$ in \eqref{A.eq:TooBig}, we see that \eqref{eq:DefnTau} implies that
$$
|\left(\prod_{j=0}^\tau \gamma_j\cdot \mathcal S\right)\cap F_N^{(d)}(\left(\bigcup_{j=0}^\tau A_j\right)^c)|> |F_N^{(d)}(\left(\bigcup_{j=0}^\tau A_j\right)^c)|
$$
reaching the desired contradiction.\hfill \qedsymbol 
\subsection{The proof of Lemma \ref{A.lem:RecurrenceLem}}
Let $\Lambda=\{D(t-1)+1,...,Dt\}$.
Suppose for contradiction that \eqref{A.eq:SmallSymmetricDiff} does not hold and for every non-empty $\beta\subseteq \Lambda$ with $|\beta|\leq d,$ set 
$$\phi(\beta):=\{(a_1,...,a_d)\in \beta^d\,|\,\forall b\in\beta\exists i\in\{1,...,d\},\,a_i=b\}.$$
We will also put $\phi(\emptyset)=\emptyset$.\\
First we show that for any collection of non-empty and distinct sets $\beta_1,...,\beta_M\subseteq \Lambda$ with $$\max_{1\leq j\leq M}|\beta_j|\leq d,$$  one has that
\begin{equation}\label{A.eq:partialBound}
\sum_{x\in X} |\sigma_{\bigcup_{j=0}^{M-1} \phi(\beta_j)}f_\mathcal Q(x)+\sigma_{\bigcup_{j=0}^M \phi(\beta_j)}f_\mathcal Q(x)|<4c_{|\beta_M|}\epsilon|X|,
\end{equation}
where $\beta_0=\emptyset$.\\
To do this, set $\Gamma=\bigcup_{j=0}^M\phi(\beta_j)\setminus \beta_M^d$ and let $\alpha_0,...,\alpha_n$ be an enumeration of the set $\{\beta_j\,|\,\beta_j\subseteq \beta_M\}$ with $\alpha_0=\emptyset$ and $\alpha_n=\beta_M$. When $|\alpha_n|=1$, we have that $n=1$ and $\alpha_n=\{a\}$ for some $a\in \Lambda$. Thus, since
$$
\sum_{x\in X}|\sigma_\Gamma f_\mathcal Q(x)+\sigma_{\phi(\alpha_1)\cup\Gamma}f_\mathcal Q(x)|+\sum_{x\in X}|\sigma_\Gamma f_\mathcal Q(x)-\sigma_{\phi(\alpha_1)\cup\Gamma}f_\mathcal Q(x)|=2|X|,
$$
we obtain by our assumption about \eqref{A.eq:SmallSymmetricDiff} that 
$$
\sum_{x\in X}|\sigma_\Gamma f_\mathcal Q(x)+\sigma_{\phi(\beta_M)\cup\Gamma}f_\mathcal Q(x)|<8\epsilon|X|=4c_1\epsilon|X|.
$$
If $d=1$, we have nothing else to show. Thus, assume $d>1$ and take $k\in\{1,...,d-1\}$. Suppose that \eqref{A.eq:partialBound} holds whenever $|\beta_M|\leq k$ and consider now the case when $|\beta_M|=k+1$. Let $\gamma_0,...,\gamma_{2^{k+1}-1}$ be an enumeration of all the subsets of $\beta_M$ with the additional property that for each $j\in\{0,...,n\}$, $\gamma_j=\alpha_j$. Note that $\bigcup_{j=1}^{2^{k+1}-1}\phi(\gamma_j)=\beta_M^d$ and that for each $r\in\{1,...,k+1\}$, 
$$
|\{\gamma_j\,|\,|\gamma_j|=r\}|=\binom{k+1}{r}.
$$
In the case that $n$ is even these observations yield that
\begingroup
\allowdisplaybreaks
\begin{multline*}
\sum_{x\in X}|\sigma_{\Gamma\cup\bigcup_{j=0}^{n-1}\phi(\gamma_j)} f_\mathcal Q(x)+\sigma_{\Gamma\cup\bigcup_{j=0}^{n}\phi(\gamma_j)}f_\mathcal Q(x)|
\leq \sum_{x\in X}|\sigma_{\Gamma\cup\bigcup_{j=0}^{n-1}\phi(\gamma_j)} f_\mathcal Q(x)+\sigma_{\Gamma}f_\mathcal Q(x)|\\
+\sum_{x\in X}|\sigma_{\Gamma} f_\mathcal Q(x)+\sigma_{\Gamma\cup\beta_M^d}f_\mathcal Q(x)|+\sum_{x\in X}|\sigma_{\Gamma\cup\beta_M^d} f_\mathcal Q(x)
+\sigma_{\Gamma\cup\bigcup_{j=0}^{n}\phi(\gamma_j)}f_\mathcal Q(x)|\\
<8\epsilon|X|+\sum_{x\in X}|\sigma_{\Gamma\cup\bigcup_{j=0}^{n-1}\phi(\gamma_j)} f_\mathcal Q(x)+\sigma_{\Gamma}f_\mathcal Q(x)|+\sum_{x\in X}|\sigma_{\Gamma\cup\beta_M^d} f_\mathcal Q(x)
+\sigma_{\Gamma\cup\bigcup_{j=0}^{n}\phi(\gamma_j)}f_\mathcal Q(x)|\\
\leq 8\epsilon|X|+\sum_{t=0,\,t\neq n-1}^{2^{k+1}-2}\left(\sum_{x\in X}|\sigma_{\Gamma\cup\bigcup_{j=0}^t\phi(\gamma_j)} f_\mathcal Q(x)
+\sigma_{\Gamma\cup\bigcup_{j=0}^{t+1}\phi(\gamma_j)}f_\mathcal Q(x)|\right)\\
=4c_0\epsilon|X|+\sum_{t=0,\,t\neq n-1}^{2^{k+1}-2}\left(\sum_{x\in X}|\sigma_{\Gamma\cup\bigcup_{j=0}^t\phi(\gamma_j)} f_\mathcal Q(x)+\sigma_{\Gamma\cup\bigcup_{j=0}^{t+1}\phi(\gamma_j)}f_\mathcal Q(x)|\right)\\
<(4c_0+4\sum_{r=1}^k \binom{k+1}{r}c_r)\epsilon|X|=4c_{k+1}\epsilon|X|.
\end{multline*}
\endgroup
When $n$ is odd, a similar argument can be used to show that 
\begin{multline*}
\sum_{x\in X}|\sigma_{\Gamma\cup\bigcup_{j=0}^{n-1}\phi(\gamma_j)} f_\mathcal Q(x)+\sigma_{\Gamma\cup\bigcup_{j=0}^{n}\phi(\gamma_j)}f_\mathcal Q(x)|
\leq \sum_{x\in X}|\sigma_{\Gamma\cup\bigcup_{j=0}^{n-1}\phi(\gamma_j)} f_\mathcal Q(x)-\sigma_{\Gamma}f_\mathcal Q(x)|\\
+\sum_{x\in X}|\sigma_{\Gamma} f_\mathcal Q(x)+\sigma_{\Gamma\cup\beta_M^d}f_\mathcal Q(x)|+\sum_{x\in X}|\sigma_{\Gamma\cup\beta_M^d} f_\mathcal Q(x)
-\sigma_{\Gamma\cup\bigcup_{j=0}^{n}\phi(\gamma_j)}f_\mathcal Q(x)|<4c_{k+1}\epsilon|X|,
\end{multline*}
which proves that \eqref{A.eq:partialBound} holds.\\
To reach the desired contradiction, notice that $\Lambda$ has only $2^D-2$ different non-empty subsets with no more than $d$ elements. Let $\beta_1,...,\beta_{2^D-2}$ be a list of these subsets. Because $\Lambda^d=\bigcup_{j=1}^{2^D-2}\phi(\beta_j)$, we have 
\begin{multline}
    \sum_{x\in X}|f_\mathcal Q(x)-\sigma_{\Lambda^d}f_\mathcal Q(x)|
    \leq \sum_{t=1}^{2^{D}-2}\sum_{x\in X}|\sigma_{\bigcup_{j=0}^{t-1}\phi(\beta_j)}f_{\mathcal Q}(x)+\sigma_{\bigcup_{j=0}^{t}\phi(\beta_j)}f_{\mathcal Q}(x)|\\
    <4\sum_{r=1}^d\binom{D}{r}c_r\epsilon|X|<4c_D\epsilon|X|=(1-4\epsilon)|X|.
\end{multline}
Thus, by our assumption about \eqref{A.eq:SmallSymmetricDiff},
$$2(1-4\epsilon)\leq \frac{1}{|X|} \sum_{x\in X}|f_\mathcal Q(x)-\sigma_{\Lambda^d}f_\mathcal Q(x)|<(1-4\epsilon), $$
a contradiction. \hfill \qedsymbol
\section{Connection between positive definite sequences and measurable recurrence}\label{B.Sec}
In this appendix we show that Conjecture \ref{0.Conjecture2} (=\cite[Conjecture 2]{berMcCuIPPolySzemeredi}) is equivalent to Conjecture \ref{0.Conjecture A} and prove Proposition \hyperlink{0.prop:UnitaryLowerBound}{E}. As we will see, both of these results are obtained with the help of the following lemma. Recall that given a countable, discrete (potentially finite) abelian group $(G,+)$, a map $\varphi:G\rightarrow \mathbb C$ is called a \textit{positive definite sequence} if for any $k\in\N$, any complex numbers $c_1,...,c_k\in\mathbb C$, and any $g_1,...,g_k\in G$,
$$
\sum_{i,j=1}^k c_j\overline{c_i}\varphi(g_j-g_i)\geq 0.
$$
\begin{lem}\label{B.lem:PositiveDefinite}
    Let $(G,+)$ be a countable, discrete abelian group and let $(\delta_g)_{g\in G}$ be a $[-1,1]$-valued  sequence satisfying $\delta_{0_G}=1$. If $(\delta_g)_{g\in G}$ is positive definite, then for every $\lambda\in (0,1]$   
    there exists an invertible probability preserving system $(X_\lambda,\mathcal{A}_\lambda,\mu_\lambda, (T_\lambda^g)_{g\in G})$ and a set $A_\lambda\in\mathcal{A}_\lambda$ with $\mu_\lambda(A_\lambda)=1/2$ such that 
    \begin{equation}\label{B.eq:MeasureCorrelation}
\frac{\sin^{-1}(\lambda\delta_g)}{2\pi}+\mathbbm 1_{\{0_G\}}(g)\frac{\pi-2\sin^{-1}(\lambda)}{4\pi}=\mu_\lambda(A_\lambda\cap T_\lambda^{-g}A_\lambda)-\mu_\lambda^2(A_\lambda)
    \end{equation}
for every $g\in G$. Furthermore,  the sequence $(\delta_g)_{g\in G}$ is positive definite if and only if  for every $\lambda\in (0,1]$ the sequence     
$$
\Phi_g^{(\lambda)}:=\frac{\sin^{-1}(\lambda\delta_g)}{2\pi},\,g\in G,
$$
is positive definite. 
 \end{lem}
 The importance of  Lemma \ref{B.lem:PositiveDefinite} for the results presented in this appendix is that it establishes that for a set $R\subseteq G$ the following three properties are equivalent:
 \begin{enumerate}[(a)]
 \item $R$ is a \textit{set of nice recurrence}:  For any invertible probability measure preserving system $(X,\mathcal A,\mu, (T^g)_{g\in G})$, any $A\in\mathcal A$, and any $\epsilon>0$, there are infinitely many $g\in R$ with  $\mu(A\cap T^{-g}A)>\mu^2(A)-\epsilon$.
     \item $R$ is a \textit{ set of $\frac{1}{2}$-nice recurrence}: For any invertible probability measure preserving system $(X,\mathcal A,\mu, (T^g)_{g\in G})$, any $A\in\mathcal A$ with $\mu(A)=1/2$, and any $\epsilon>0$, there are infinitely many $g\in R$ with  $\mu(A\cap T^{-g}A)>\mu^2(A)-\epsilon$.
     \item For any unitary action $(U^g)_{g\in G}$ of $G$ on a Hilbert space $\mathcal H$,  any $f\in\mathcal H$, and any $\epsilon>0$, there are infinitely many $g\in R$ such that $\Re(\langle f,U^g f\rangle)>-\epsilon$. 
 \end{enumerate}
 In Subsection B.1 we will show that Conjecture \ref{0.Conjecture2} is equivalent to Conjecture \ref{0.Conjecture A}. In Subsection B.2 we prove Proposition \hyperlink{0.prop:UnitaryLowerBound}{E}. In Subsection B.3 we prove Lemma \ref{B.lem:PositiveDefinite}.
 \subsection{The equivalent forms of  \cite[Conjecture 2]{berMcCuIPPolySzemeredi}}\label{B.Sec:Conjecture2}
 The following lemma is instrumental in proving the equivalence of  Conjectures \ref{0.Conjecture A} and \ref{0.Conjecture2}.
 \begin{lem}\label{B.lem:NiceInjectiveRestrictionForVIP}
     Let $\varphi:\FF\rightarrow\FF$ be an injective homomorphism. There is an injective homomorphism $\psi:\FF\rightarrow \FF$ with the property that for any non-zero $\xi,\eta\in\FF$, if $\max \alpha_\xi<\min \alpha_\eta$, then $\max \alpha_{\psi(\xi)}<\min \alpha_{\psi(\eta)}$ and  $\max \alpha_{\varphi(\psi(\xi))}<\min \alpha_{\varphi(\psi(\eta))}$.
 \end{lem}
 \begin{proof}[Proof of Lemma \ref{B.lem:NiceInjectiveRestrictionForVIP}]
     Note that by the pigeonhole principle, for any $(2^{k}+1)$-element subset $\mathcal S$ of $\FF\setminus\{0_\FF\}$, one can find  distinct $\xi,\eta\in \mathcal S$ with $$\alpha_{\varphi(\xi)}\cap \{1,...,k\}=\alpha_{\varphi(\eta)}\cap\{1,...,k\}$$ and, so, $\min (\alpha_{\varphi(\xi+\eta)})>k$. It follows that we can pick a sequence $(\xi_j)_{j\in\N}$  in $\FF\setminus\{0_\FF\}$  with the properties that for each $j\in\N$, $\max \alpha_{\xi_j}<\min \alpha_{\xi_{j+1}}$ and 
   $\max \alpha_{\varphi(\xi_j)}<\min \alpha_{\varphi(\xi_{j+1})}$. The result follows by letting $\psi$ be the only homomorphism satisfying $\psi(\mathbbm 1_{\{j\}})=\xi_j$ for each $j\in\N$. 
 \end{proof}
We now move to proving that  Conjecture \ref{0.Conjecture2} implies Conjecture \ref{0.Conjecture A}. Consider a VIP-system $v:\FF\rightarrow G$, an invertible probability measure preserving system $(X,\mathcal A,\mu,(T^g)_{g\in G})$, and any $A\in\mathcal A$.  For each $\xi\in\FF$, let $V_\xi=T^{v(\xi)}$ and set $f_A=\mathbbm 1_A-\mu(A)$. By Hindman's theorem (see Theorem \ref{0.thm:Hindman} above), we can find an injective homomorphism $\varphi:\FF\rightarrow \FF$ such that for every $f\in L^2(\mu)$, $\mathop{\text{{\rm IP-lim}}}_{\xi\in\FF} V_{\varphi(\xi)}f=Qf$ exists
(here we use the fact that when dealing with a sequence of the form $(\mu(A\cap T^g A))_{g\in G}$, one can always  assume that $L^2(\mu)$ is separable). 
Invoking Lemma \ref{B.lem:NiceInjectiveRestrictionForVIP} and  Remark \ref{0.VIPficationRemark}, we can assume without loss of generality that $(V_{\varphi(\xi)})_{\xi\in\FF}$ is a VIP-system. Thus, by  Conjecture \ref{0.Conjecture2},
 \begin{multline*}
\mathop{\text{{\rm IP-lim}}}_{\xi\in\FF} \mu(A\cap T^{-v(\varphi(\xi))}A)=\mathop{\text{{\rm IP-lim}}}_{\xi\in\FF}\langle (\mu(A)+f_A),T^{v(\varphi(\xi))}(\mu(A)+f_A) \rangle\\
=\mathop{\text{{\rm IP-lim}}}_{\xi\in\FF}\langle \mu(A),T^{v(\varphi(\xi))}\mu(A) \rangle+\mathop{\text{{\rm IP-lim}}}_{\xi\in\FF}\langle f_A,T^{v(\varphi(\xi))}f_A\rangle=\mu^2(A)+\langle f_A,Qf_A \rangle\geq \mu^2(A). 
 \end{multline*}
It follows that for every $\epsilon>0$ there are infinitely many non-zero $\xi\in\FF$ with $\mu(A\cap T^{-v(\xi)}A)>\mu^2(A)-\epsilon$ as guaranteed by Conjecture \ref{0.Conjecture A}.\\
To prove that Conjecture \ref{0.Conjecture A} implies Conjecture \ref{0.Conjecture2}, let $(V_\xi)_{\xi\in\FF}$ be a VIP-system of unitary operators acting on a Hilbert space $\mathcal H$ with the property that for every $f\in\mathcal H$ there is a $Qf\in\mathcal H$ with 
$$
\mathop{\text{{\rm IP-lim}}}_{\xi\in\FF} V_{\xi}f=Qf.
$$
Let $\mathcal V$ be the countable abelian group generated by the commuting family of unitary maps $V_\xi$, $\xi\in\FF$. Take $f\in\mathcal H$ with $\|f\|_\mathcal H=1$ and note that the sequence 
$$\delta_V:=\Re(\langle f, Vf\rangle),\,V\in\mathcal V$$
is a $[-1,1]$-valued positive definite sequence. By Lemma \ref{B.lem:PositiveDefinite}, there is an invertible probability preserving system $(X,\mathcal A,\mu,(T^V)_{V\in\mathcal V})$ and a set $A\in\mathcal A$ with $\mu(A)=1/2$ for which \eqref{B.eq:MeasureCorrelation} holds with $\lambda=1$. Let $k\in\N$.  Applying Conjecture \ref{0.Conjecture A} to the VIP-system $(V_\xi)_{\min \alpha_\xi>k}$, we see that for every $\epsilon\in(0,1/4)$, there is a non-zero $\xi\in\FF$ with $\min \alpha_\xi>k$  such that
$$\frac{1}{4}+\frac{\sin^{-1}(\delta_{V_{\xi}})}{2\pi}=\mu(A\cap T^{-V_{\xi}}A)>\frac{1}{4}-\epsilon,
$$
which implies that $\Re(\langle f,V_{\xi}f\rangle)=\delta_{V_{\xi}}>\sin(-2\pi\epsilon)$ and, so, $\Re(\langle f,Qf\rangle)\geq 0$. We are done. 
 \subsection{The proof of Proposition E}\label{B.sec:prooofOfprop}
 Fix $d\in\N$. Combining \cref{A.thm:EpsilonPDHJ} with  Proposition \ref{0.prop.StrongRec}, we obtain the existence of a constant $c_{d,1/2}>0$ for which formula \eqref{0.eq:DPHJLowerBound} holds.  
 Consider now a VIP-system $(V_\xi)_{\xi\in\FF}$ of degree at most $d$  taking values in an abelian group of unitary operators acting on a Hilbert space $\mathcal H$ and let $f\in\mathcal H$ be such that $\|f\|_\mathcal H=1$. As before, we let $\mathcal V$ denote the countable abelian group generated by the unitary maps $V_\xi$, $\xi\in\FF$, and note that  
$$\delta_V:=\Re(\langle f, Vf\rangle),\,V\in\mathcal V$$ 
is a $[-1,1]$-valued positive definite sequence. By  Lemma \ref{B.lem:PositiveDefinite}, there is an invertible probability preserving system $(X,\mathcal A,\mu,(T^V)_{V\in\mathcal V})$ and a set $A\in\mathcal A$ with $\mu(A)=1/2$ such that for every $\xi\in\FF$, 
$$
\frac{1}{4}+\frac{\sin^{-1}(\delta_{V_\xi})}{2\pi}=\mu(A\cap T^{-V_\xi}A).
$$
It now follows that for some non-zero $\xi\in\FF$, $\mu(A\cap T^{-V_\xi}A)\geq \frac{c_{d,1/2}}{2}$ and, so, 
$$
\Re(\langle f,V_\xi f\rangle)=\delta_{ V_\xi}\geq \sin(-\frac{\pi}{2}+\pi c_{d,1/2})=-\cos(\pi c_{d,1/2}).
$$
Setting $\epsilon_d:=1-\cos(\pi c_{d,1/2})$ and noting that $(V_\xi)_{\xi\in\FF}$ was arbitrary, we see that Remark \ref{0.VIPficationRemark} implies that whenever 
$$
\mathop{\text{{\rm IP$^*$-lim}}}_{\xi\in\FF}\;\langle f,W_{\varphi(\xi)}f\rangle
$$
exists for some vector $f$, some VIP-system $(W_\xi)_{\xi\in\FF}$ of degree at most $d$,  and some injective homomorphism $\varphi:\FF\rightarrow \FF$, one has that   
$$
\mathop{\text{{\rm IP$^*$-lim}}}_{\xi\in\FF}\;\Re(\langle f,W_{\varphi(\xi)}f\rangle)\geq \|f\|_\mathcal H^2(-1+\epsilon_d).
$$
We are done. 
\subsection{The proof of Lemma \ref{B.lem:PositiveDefinite}}
The proof of Lemma \ref{B.lem:PositiveDefinite} makes use of a class of probability measure preserving systems known as Gaussian systems. Given a countable, discrete abelian group $(G,+)$, we will let 
 $\mathcal B$ denote the Borel $\sigma$-algebra on $\R^G$ and for each $g\in G$, we will let $X_g:\R^G\rightarrow \R$ denote the canonical projection onto the $g$-th coordinate (i.e. $X_g(\omega)=\omega(g))$. For any real-valued  positive definite sequence $\delta:G\rightarrow \R$, we define the Gaussian measure associated with $\delta$ to be the unique probability measure $\gamma_\delta$ on $\mathcal B$ defined by the property that for any $N\in\N$, any Borel measurable sets $A_1,...,A_N\subseteq \R$, and any $g_1,...,g_N\in G$, one has that 
$$
\gamma_\delta(\bigcap_{j=1}^N X_{g_j}^{-1}A_j)=\mathbb P((\mathcal Z_1,...,\mathcal Z_N)\in \prod_{j=1}^NA_j),
$$
where the random variables $\mathcal Z_1,...,\mathcal Z_N$  have joint Gaussian distribution with covariance  matrix $\Sigma$ given by 
$$
\Sigma_{i,j}=\delta_{g_i-g_j},\,i,j\in\{1,...,N\}
$$
and for each $j\in\{1,...,N\}$, $\mathbb E(\mathcal Z_j)=0$.
Note that because $\delta$ is real-valued, the matrix $\Sigma$ is always symmetric.\\
Under the additional assumption that $\delta_{0_G}=1$, \cite[Lemma 2.2]{ZelIP0Khintchine2023} implies that for any $g\in G$,
\begin{equation}\label{B.eq:IntegralAtTimeg}
\gamma_\delta(X_{0_G}^{-1}[0,\infty)\cap X_g^{-1}[0,\infty))=\frac{1}{4}+\frac{\sin^{-1}(\delta_g)}{2\pi}.
\end{equation}
For each $g\in G$, denote the map $T^g:\R^G\rightarrow \R^G$ by $[T^g\omega](h)=\omega(h+g)$. The quadruple $(\R^G,\mathcal B,\gamma_\delta, (T^g)_{g\in G})$ forms an invertible probability preserving system known as the Gaussian system associated with $\delta$. See \cite{cornfeld1982ergodic},\cite{kechris2010Global},\cite{zelada2023GaussianETDS}, for example, for further discussion on Gaussian systems.
\begin{proof}[Proof of Lemma \ref{B.lem:PositiveDefinite}]
First assume that $(\delta_g)_{g\in G}$ is positive definite and consider the positive definite sequence $(\mathbbm 1_{\{0_G\}}(g))_{g\in G}$. Observe that for any $\lambda\in (0,1]$, the sequence 
$$\delta^{(\lambda)}_g:=\lambda \delta_g+(1-\lambda)\mathbbm 1_{\{0_G\}}(g),\, g\in G,$$
is positive definite. Let $(X_\lambda,\mathcal A_\lambda,\mu_\lambda, (T_\lambda^g)_{g\in G})$ be the Gaussian system associated with $\delta^{(\lambda)}$. Setting $A_\lambda=X_{0_G}^{-1}[0,\infty)$, we see that formula \eqref{B.eq:IntegralAtTimeg} implies that for any  $g\in G$,
$$
\mu_\lambda(A_\lambda\cap T_\lambda^{-g}A_\lambda)=\frac{1}{4}+\frac{\sin^{-1}(\delta^{(\lambda)}_g)}{2\pi}.
$$
Thus, for any $g\in G$,
\begin{multline*}
\frac{\sin^{-1}(\lambda \delta_g)}{2\pi}=\frac{\sin^{-1}(\delta^{(\lambda)}_g)}{2\pi}-\mathbbm 1_{\{0_G\}}(g)\frac{\pi-2\sin^{-1}(\lambda)}{4\pi}\\
=
\mu_\lambda(A_\lambda\cap T_\lambda^{-g}A_\lambda)-\frac{1}{4}-\mathbbm 1_{\{0_G\}}(g)\frac{\pi-2\sin^{-1}(\lambda)}{4\pi},
\end{multline*}
as claimed.\\
That $\Phi_g^{(\lambda)}=\frac{\sin^{-1}(\lambda\delta_g)}{2\pi}$, $g\in G$, is positive definite follows from the fact that $\sin^{-1}(x)=\sum_{n=1}^\infty a_nx^n$ for some sequence of non-negative numbers $(a_n)_{n\in\N}$ and   
Schur's product theorem (see \cite[Theorem 7.5.3]{HornJohnson-MatrixAnalysis-1985}, for example) which implies that for any positive definite sequence $(c_g)_{g\in G}$ and any $k\in\N$, $(c_g^k)_{g\in G}$ is also positive definite (the latter fact can also be proved by invoking Bochner's theorem and using the fact  that the Fourier transform of a convolution of two finite Borel measures is the product of their corresponding Fourier transforms).  
\\
For the opposite direction, simply note that for every $g\in G$,
$$\delta_g=\lim_{\lambda\rightarrow 0^+}\frac{\sin^{-1}(\lambda \delta_g)}{\lambda}.$$
We are done. 
\end{proof}
\section{ Equivalent forms of the Furstenberg-S{\'a}rk{\"o}zy theorem for $\FF$-valued polynomials}\label{C.Sec}
Our goal in this appendix is to prove the following result which establishes the equivalence of various results of combinatorial and dynamical nature (including the Furstenberg-S{\'a}rk{\"o}zy theorem for $\FF$-valued polynomials). We remark that item (i) in Theorem \ref{C.thm:EquivalentForms} is a weaker form of the Density Polynomial Hales-Jewett conjecture mentioned in \cite{GowersDPHJjBlog}. Notice that the special case of Proposition \ref{0.prop.StrongRec} for $\FF$-valued VIP-systems follows from (i)$\implies$(ii) below.
\begin{thm}\label{C.thm:EquivalentForms}
Let $d\in\N$ and $\delta\in (0,1)$. The following statements are equivalent:
\begin{enumerate}[(i)]
    \item There is an $r:=r_{d,\delta}\in\N$ with the property that for any $N\in\N$ with $N\geq r$ and any $\mathcal S\subseteq \mathcal P(\{1,...,N\}^d)$ with $|\mathcal S|\geq \delta 2^{N^d}$ one can find $E,F\in\mathcal S$ and a non-empty $\gamma\subseteq \{1,...,N\}$ such that $E\triangle F=\gamma^d$.
    \item There is an $r:=r_{d,\delta}\in\N$ with the property that for  any VIP-system $v:\mathcal F\rightarrow \FF$ of degree at most $d$, any invertible probability preserving system $(X,\mathcal A,\mu, (T^\xi)_{\xi\in\FF})$, any $A\in\mathcal A$ with $\mu(A)\geq\delta$,  and any non-empty pairwise disjoint $\alpha_1,...,\alpha_r\in\mathcal F$, one can find $1\leq i_1<\cdots<i_t\leq r$ such that 
    $$
\mu(A\cap T^{-v(\alpha_{i_1}\cup\cdots\cup \alpha_{i_t})}A)>\frac{1}{2^{r^d+r}}\frac{\mu(A)-\delta}{1-\delta}.
    $$
     \item There is  an $r:=r_{d,\delta}\in\N$ with the property that for  any infinite abelian group $G$, any polynomial  $p:G\rightarrow \FF$ satisfying $p(0_G)=0_\FF$ of degree at most $d$, any invertible probability preserving system $(X,\mathcal A,\mu, (T^\xi)_{\xi\in\FF})$, any $A\in\mathcal A$ with $\mu(A)\geq\delta$,  and any $g_1,...,g_r\in G$  one can find $1\leq i_1<\cdots<i_t\leq r$ such that 
    $$
\mu(A\cap T^{-p(g_{i_1}+\cdots+g_{i_t})}A)>\frac{1}{2^{r^d+r}}\frac{\mu(A)-\delta}{1-\delta}.
    $$
    So, in particular, the set $\{g\in G\,|\,\mu(A\cap T^{-p(g)}A)>\frac{1}{2^{r^d+r}}\frac{\mu(A)-\delta}{1-\delta}\}$ is syndetic.
    \item The polynomial $p_d:\FF\rightarrow \FF\cong\mathcal P_d\times \FF$ defined in Section \ref{2.Sec} (see \eqref{2.eq:Defnp_d})
    has the property that for any invertible probability preserving system $(X,\mathcal A,\mu, (T^\xi)_{\xi\in\FF})$ and any $A\in\mathcal A$ with $\mu(A)\geq\delta$, one can find infinitely many $\xi\in \FF$ such that 
    $$
\mu(A\cap T^{-p_d(\xi)}A)>0.
    $$
    \item The polynomial $q_d:\FF\rightarrow \FF\cong\mathcal P_d$ defined in \eqref{2.eq:Defnq_d}
    has the property that for any invertible probability preserving system $(X,\mathcal A,\mu, (T^\xi)_{\xi\in\FF})$ and any $A\in\mathcal A$ with $\mu(A)\geq\delta$, one can find infinitely many $\xi\in \FF$ such that 
    $$
\mu(A\cap T^{-q_d(\xi)}A)>0.
    $$
\end{enumerate}
\end{thm}
The implication (iii)$\implies$(iv) is immediate. That (ii)$\implies$(iii) can be proved by choosing  a sequence $(h_k)_{k\in\N}$ in $G$ with the property that $h_j=g_j$ for each $j\in\{1,...,r\}$ and then defining the IP-polynomial $\alpha\mapsto p(h_\alpha)$, where $h_\alpha:=\sum_{j\in\alpha}h_j$ (the syndeticity of the set 
$$\{g\in G\,|\,\mu(A\cap T^{-p(g)}A)>\frac{1}{2^{r^d+r}}\frac{\mu(A)-\delta}{1-\delta}\}$$
follows from the fact that it has a non-trivial intersection with every set of the form $\{h_\alpha\,|\,\alpha\in\mathcal F\}$).
That (iv)$\implies$(v) follows by noting that any $\mu$-preserving  $\mathcal P_d$-action $(T^{\ \xi})_{\xi\in\mathcal P_d}$ can be  viewed as an action of $\mathcal P_d\times \FF$ by letting $S^{(\xi,\eta)}:=T^\xi$ for every $(\xi,\eta)\in \mathcal P_d\times \FF$. Thus, all that we need to show is that (i)$\implies$(ii) and that (v)$\implies$(i). 
\subsection{The proof of (i)$\implies$(ii)}
Our proof of (i)$\implies$(ii) is an easy modification of the proof of Theorem 6.15 in \cite{berMcCuIPPolySzemeredi}. We will need the following lemma which was originally utilized in \cite{BerLeibPolyHJ} (see \cite[Lemma A.2]{BerZel-NiceRecurrence} for a proof). For each $D\in\N$, $\mathcal F_\emptyset(\N^D)$ denotes the family of all  finite subsets of $\N^D$.
\begin{lem}\label{C.lem:ShiftToSetPolynomial}
    Let  $D\in\N$ and let $(G,+)$ be an abelian group. For any VIP-system $(g_\alpha)_{\alpha\in\mathcal F}$ in $G$ of degree at most $D$ there is a function $\eta:\mathcal F_\emptyset (\N^D)\rightarrow G$ with the following two properties:
    \begin{enumerate}
       \item [(P.1)] For any $\alpha\in\mathcal F$,    $$\eta(\alpha^D)=\eta(\underbrace{\alpha\times\cdots\times\alpha}_{D\text{ times}})=g_\alpha.$$
         \item [(P.2)] For any $E,F\in\mathcal F_\emptyset (\N^D)$ with $E\cap F=\emptyset$, $\eta(E\cup F)=\eta(E)+\eta(F)$.
    \end{enumerate}
\end{lem}
\begin{proof}[Proof of (i)$\implies$(ii) in Theorem \ref{C.thm:EquivalentForms}]
Let $r:=r_{d,\delta}$ be as in the statement of (i).
Fix a VIP-system $v:\mathcal F\rightarrow\FF$ of degree at most $d$, let $\alpha_1,...,\alpha_r$ be non-empty, finite pairwise disjoint subsets of $\N$, let $(X,\mathcal A,\mu,(T^\xi)_{\xi\in\FF})$ be an invertible probability preserving system, and let $A\in\mathcal A$ be such that $\mu(A)\geq \delta$. If $\mu(A)=1$, then there is nothing to prove. Thus assume $\mu(A)\in (0,1)$.\\
By \cref{C.lem:ShiftToSetPolynomial}, there is a map $\eta:\mathcal F_\emptyset (\N^d)\rightarrow \FF$ with the properties that for every $\alpha\in\mathcal F$, $v(\alpha)=\eta(\alpha^d)$ and for any disjoint $E,F\in \mathcal F_\emptyset(\N^d)$, $\eta(E\cup F)=\eta(E)+\eta(F)$. For any $\beta\subseteq \{1,...,r\}^d$, let 
$$
\varphi(\beta)=\prod_{(j_1,...,j_d)\in\beta}T^{-\eta(\alpha_{j_1}\times\cdots\times \alpha_{j_d})}.
$$
Let 
$$\mathcal M=\{\beta\,|\,\beta\subseteq \{1,...,r\}^d\}.$$
We define the function $f:X\rightarrow [0,1]$ by
$$f(x)=\frac{1}{|\mathcal M|}\sum_{\beta\in\mathcal M} \mathbbm 1_{\varphi(\beta)A}(x).$$
For each $x\in X$, we set 
$$\mathcal S_x=\{\beta\in\mathcal M\,|\,\mathbbm 1_{\varphi(\beta)A}(x)=1\}.$$
Note that $|\mathcal S_x|=|\mathcal M|f(x)$. Let $Y:=\{x\in X\,|\,f(x)\geq\delta\}$. For any $x\in Y$,  $\frac{|\mathcal S_x|}{|\mathcal M|}\geq \delta$ and, so, by (i), there exist $E_{x},F_x\in\mathcal S_x$
and a non-empty $\gamma_x\subseteq \{1,...,r\}$ such that $E_x\triangle F_x=\gamma_x^d$. It follows that 
\begin{equation}\label{C.eq:UpperBoundForY}
Y\subseteq\bigcup_{E\in \mathcal M}\bigcup_{\gamma\subseteq \{1,...,r\},\,\gamma\neq\emptyset}\bigcup_{F\in\mathcal M,\,E\triangle F=\gamma^d}\{x\in X\,|\,\mathbbm 1_{\varphi(E)A}(x)\mathbbm 1_{\varphi(F)A}(x)=1\}.
\end{equation}
Noting that in the right-hand side of \eqref{C.eq:UpperBoundForY} there are no more than $2^r$ choices for $\gamma$ and, once $\gamma$ is selected, there are no more than  $2^{r^d}$ choices for the pair $(E,F)$, we see that for some
 $E_0,F_0\in \mathcal M$ and a non-empty $\gamma_0\subseteq \{1,...,r\}$ with $E_0\triangle F_0=\gamma_0^d$, the set 
 $$Z=\{x\in X\,|\,\mathbbm 1_{\varphi(E_0)A}(x)\mathbbm 1_{\varphi(F_0)A}(x)=1\}$$
 satisfies $\frac{\mu(Y)}{2^{r^d+r}}\leq \mu(Z)$.\\ 
To conclude the proof, note that since $\int_Xf(x)\text{d}
\mu(x)=\mu(A)$ and $0\leq f(x)\leq 1$ for each $x\in X$, we must have
$$\delta\leq\mu(A)=\int_X f(x)\text{d}\mu(x)<\mu(Y)+\delta(1-\mu(Y))$$
which yields 
$$
\mu(Y)> \frac{\mu(A)- \delta}{1-\delta}.
$$
 Thus, 
 \begin{multline*}
     \mu(A\cap T^{-v(\bigcup_{j\in\gamma_0}\alpha_j)}A)
     =\mu(A\cap T^{-\eta((\bigcup_{j\in\gamma_0}\alpha_j)^d)}A)
=\mu(\varphi(E_0\setminus F_0)A\cap \varphi(F_0\setminus E_0)A)\\
     =\mu(\varphi(E_0\cap F_0)\Big(\varphi(E_0\setminus F_0)A\cap \varphi(F_0\setminus E_0)A\Big))=\mu(\varphi(E_0)A\cap \varphi(F_0)A)\\
     \geq \mu(Z)> \frac{1}{2^{ r^d+r}}\frac{
     \mu(A)-\delta}{1-\delta}.
 \end{multline*} 
 We are done. 
\end{proof}
\subsection{The proof of (v)$\implies$(i)}
Our proof of the implication (v)$\implies$(i) in Theorem \ref{C.thm:EquivalentForms} makes use of the following variant of a result due to Forrest \cite{forrest1990recurrence}. For any countable abelian group $G$ and any $\delta\in (0,1)$, we say that a set $R\subseteq G$ is a set of $\delta$-recurrence if for any invertible probability measure preserving system $(X,\mathcal A,\mu,(T^g)_{g\in G})$ and any $A\in\mathcal A$ with $\mu(A)\geq \delta$, one has that there is a non-zero $g\in R$ with 
$$\mu(A\cap T^{-g}A)>0.$$
\begin{lem}[Cf. Lemma 6.4 in \cite{forrest1990recurrence}] \label{lem:forrest-ackelsberg}
Let $(G,+)$ be a countable abelian group, let $R\subseteq G$, and let $\delta\in(0,1)$. If $R$ is a set of $\delta$-recurrence, then there is a finite subset $R_\delta\subseteq R$ which is a set of $\delta$-recurrence.
\end{lem}
\begin{proof}[Proof of Lemma \ref{lem:forrest-ackelsberg}]
 By analyzing the  proof of  \cite[Lemma 6.4]{forrest1990recurrence}, one sees that for every $\delta\in (0,1)$, the conclusion of  \cite[Lemma 6.4]{forrest1990recurrence} still holds when one assumes only that $R$ is a set of $\delta$-recurrence as defined in this paper. Thus,  Lemma \ref{lem:forrest-ackelsberg} is an immediate consequence of the proof of \cite[Lemma 6.4]{forrest1990recurrence}.
\end{proof}
\begin{proof}[Proof of (v)$\implies$(i) in Theorem \ref{C.thm:EquivalentForms}]
For each $N\in\N$, we identify the set $\mathcal P(\{1,...,N\}^d)$ with the set of functions $X_N:=\{-1,1\}^{N^d}$ via the map $A\in \mathcal P(\{1,...,N\}^d)\mapsto (1-2\mathbbm 1_A)\in X_N$ and let $\mu_N$ denote the normalized counting measure on $X_N$. We also identify $\mathcal P_d$ with the class of all finite subsets of $\N^d$ (which we denote by $\mathcal F_\emptyset(\N^d)$) and for each $N\in\N$ define the $\mu_N$-preserving $\FF$-action $(T^{\Gamma})_{\Gamma\in \mathcal F_\emptyset(\N^d)}$ by 
$$T^\Gamma (1-2\mathbbm 1_A)=(1-2\mathbbm 1_{\Gamma\cap \{1,...,N\}^d})(1-2\mathbbm 1_A).$$
Observe that under the identification $\mathcal P_d\cong \mathcal F_\emptyset(\N^d)$ one has that for any $\xi\in\FF$, $q_d(\xi)=\alpha_\xi^d$. Thus, since  $q_d(\FF)$  is a set of $\delta$-recurrence, Lemma \ref{lem:forrest-ackelsberg} implies that for some $M\in\N$, the set $\{q_d(\xi)\,|\,\alpha_\xi\subseteq \{1,...,M\}\}$ is a set of $\delta$-recurrence. 
Thus, for $N\geq M$ and $\mathcal S\subseteq X_N$ with $\mu_N(\mathcal S)\geq \delta$, one has that there is a non-zero $\xi\in\FF$ with $\alpha_\xi\subseteq \{1,...,M\}$ and $\mu_N(\mathcal S\cap T^{-q_d(\xi)}\mathcal S)>0$. Thus, for some $E,F\in \mathcal S$, 
$$(1-2\mathbbm 1_{\alpha_\xi^d})(1-2\mathbbm 1_E)=(1-2\mathbbm 1_F),$$ 
which yields 
$$(1-2\mathbbm 1_{\alpha_\xi^d})=(1-2\mathbbm 1_E)(1-2\mathbbm 1_F).$$
 We are done. 
\end{proof}
 \bibliography{Bib.bib}

@article {ABBLargeIntersection2021,
    AUTHOR = {Ackelsberg, E. and Bergelson, V. and Best, A.},
     TITLE = {Multiple recurrence and large intersections for abelian group
              actions},
  JOURNAL = {Discrete Analysis},
      YEAR = {2021},
}

@article{AckBer2025RingsOfIntegers,
  title={Polynomial actions of rings of integers of global fields and quasirandomness of {P}aley-type graphs},
  author={Ackelsberg, E. and Bergelson, V.},
  journal={arXiv preprint arXiv:2509.17868},
  year={2025}
}

@inproceedings{ERTaU,
  title={Ergodic {R}amsey theory---{A}n {U}pdate},
  author={Bergelson, V.},
  booktitle={Ergodic {T}heory of $\mathbb{Z}^d$-{A}ctions},
   pages={1--61},
  organization={London Mathematical Society Lecture Note Series}, 
  volume={228},
  publisher={Cambridge University Press},
  address={Warwick, 1993-1994},
  year={1996}
}

@article{MultifariousPoincare,
  title={The Multifarious {P}oincar{\'e} Recurrence Theorem},
  author={Bergelson, V.},
  journal={Descriptive Set Theory and Dynamical Systems},
  pages={31--57},
  year={2000},
  publisher={Cambridge University Press, New York}
}

@article{BFM,
  title={{IP}\text{-}Sets and Polynomial Recurrence},
  author={Bergelson, V. and Furstenberg, H. and McCutcheon, R.},
  journal={Ergodic Theory and Dynamical Systems},
  volume={16},
  number={5},
  pages={963--974},
  year={1996},
  publisher={Cambridge University Press}
}

@article {BHKNilSystems2005,
    AUTHOR = {Bergelson, V. and Host, B. and Kra, B.},
     TITLE = {Multiple recurrence and nilsequences},
      NOTE = {With an appendix by Imre Ruzsa},
  JOURNAL = {Inventiones Mathematicae},
    VOLUME = {160},
      YEAR = {2005},
    NUMBER = {2},
     PAGES = {261--303}
}

@article{BerLeibPolyHJ,
    AUTHOR = {Bergelson, V. and Leibman, A.},
     TITLE = {Set-polynomials and polynomial extension of the
              {H}ales-{J}ewett theorem},
  JOURNAL = {Annals of Mathematics. Second Series},
    VOLUME = {150},
      YEAR = {1999},
    NUMBER = {1},
     PAGES = {33--75},
}

@article{BHM,
  title={{I}{P}\text{-}Systems, Generalized Polynomials and Recurrence},
  author={Bergelson, V. and Knutson H{\aa}land, I. and McCutcheon, R.},
  journal={Ergodic Theory and Dynamical Systems},
  volume={26},
  number={4},
  pages={999--1019},
  year={2006},
  publisher={Cambridge University Press},
}

@article{berMcCuIPPolySzemeredi,
  title={An ergodic {I}{P} polynomial {S}zemer{\'e}di theorem},
  author={Bergelson, V. and McCutcheon, R.},
  journal={Memoirs of the American Mathematical Society},
  volume={146},
  number={695},
  year={2000},
  note={106 pp.},
  publisher={American Mathematical Society}
}

@article{BDonaldRobertsonIP_r,
  title={Polynomial recurrence with large intersection over countable fields},
  author={Bergelson, V. and Robertson, D.},
  journal={Israel Journal of Mathematics},
  volume={214},
  number={1},
  pages={109--120},
  year={2016},
 publisher={Springer},
}

@article {BerRos1988-MixingForGroups,
    AUTHOR = {Bergelson, V. and Rosenblatt, J.},
     TITLE = {Mixing actions of groups},
  JOURNAL = {Illinois Journal of Mathematics},
    VOLUME = {32},
      YEAR = {1988},
    NUMBER = {1},
     PAGES = {65--80},
}

@article{BerGor2005-WeakMixing,
  title={Weakly mixing group actions: a brief survey and an example},
  author={Bergelson, V. and Gorodnik, A.},
  journal={arXiv preprint math/0505025},
  year={2005}
}

@book{cornfeld1982ergodic,
  title={Ergodic theory},
  author={Cornfeld, I. P.  and Fomin, S. V. and Sinai, Ya. G.},
  series={{G}rundlehren der {M}athematischen {W}issenschaften}, 
  volume={245},
  year={1982},
  publisher={Springer}
}

@book{forrest1990recurrence,
  title={Recurrence in dynamical systems: a combinatorial approach},
  author={Forrest, A. H.},
  year={1990},
  publisher={The Ohio State University}
}

@book{FBook,
   title={Recurrence in ergodic theory and combinatorial number theory},
  author={Furstenberg, H.},
  year={1981},
  publisher={Princeton University Press},
  pages={xi+203}
}

@article{FKIPSzemerediLong,
  title={An ergodic {S}zemer{\'e}di theorem for {I}{P}-systems and combinatorial theory},
  author={Furstenberg, H. and Katznelson, Y.},
  journal={Journal d'Analyse Math{\'e}matique},
  volume={45},
  pages={117--168},
  year={1985},
  publisher={Springer}
}

@article {FurKatzDHJ,
    AUTHOR = {Furstenberg, H. and Katznelson, Y.},
     TITLE = {A density version of the {H}ales-{J}ewett theorem},
  JOURNAL = {Journal d'Analyse Math\'ematique},
    VOLUME = {57},
      YEAR = {1991},
     PAGES = {64--119},
}

@misc{GowersDPHJjBlog,
    author={Gowers, T.},
  title = {The first unknown case of polynomial {D}{H}{J}},
  howpublished = {\url{https://gowers.wordpress.com/2009/11/14/the-first-unknown-case-of-polynomial-dhj/}},
  note = {Accessed: 2023-10-15}
}

@article {HalesJewett1963,
    AUTHOR = {Hales, A. W. and Jewett, R. I.},
     TITLE = {Regularity and positional games},
  JOURNAL = {Transactions of the American Mathematical Society},
    VOLUME = {106},
      YEAR = {1963},
     PAGES = {222--229},
}

@article{HIPPartitionRegular,
  title={Finite sums from sequences within cells of a partition of $\mathbb{N}$},
  author={Hindman, N.},
  journal={Journal of Combinatorial Theory, Series A},
  volume={17},
  pages={1--11},
  year={1974},
  publisher={Elsevier}
}

@book{HornJohnson-MatrixAnalysis-1985, 
title={Matrix Analysis}, 
publisher={Cambridge University Press}, 
author={Horn, R. A. and Johnson, C. R.}, 
year={1985},
}

@article {KamaeFranceNiceRec1978,
    AUTHOR = {Kamae, T. and Mend{\`e}s France, M.},
     TITLE = {Van der {C}orput's difference theorem},
  JOURNAL = {Israel Journal of Mathematics},
    VOLUME = {31},
      YEAR = {1978},
    NUMBER = {3-4},
     PAGES = {335--342},
}

@article {MR5029922,
    AUTHOR = {Karam, T.},
     TITLE = {Unions of intervals in codes based on powers of sets},
  JOURNAL = {European Journal of Combinatorics},
    VOLUME = {134},
      YEAR = {2026},
     PAGES = {Paper No. 104350, 10},
}

@book{kechris2010Global,
  title={Global aspects of ergodic group actions},
  author={Kechris, A. S.},
  year={2010},
  publisher={American Mathematical Society}
}

@article{leibman1998polynomial,
  title={Polynomial sequences in groups},
  author={Leibman, A.},
  journal={Journal of Algebra},
  volume={201},
  number={1},
  pages={189--206},
  year={1998},
}

@article{mccutcheon2005fvip,
  title={{F}{V}{I}{P} systems and multiple recurrence},
  author={McCutcheon, R.},
  journal={Israel Journal of Mathematics},
  volume={146},
  pages={157--188},
  year={2005},
  publisher={Springer}
}

@article{McCutWSarkozy2014,
  title={{D} sets and a {S}{\'a}rk{\"o}zy theorem for countable fields},
  author={McCutcheon, R. and Windsor, A.},
  journal={Israel Journal of Mathematics},
  volume={201},
  number={1},
  pages={123--146},
  year={2014},
  publisher={Springer}
}

@article{sarkozy1978difference,
  title={On difference sets of sequences of integers. {I}{I}{I}},
  author={S{\'a}rk{\" o}zy, A.},
  journal={Acta Mathematica Hungarica},
  volume={31},
  number={1-2},
  pages={125--149},
  year={1978},
  publisher={Akad{\'e}miai Kiad{\'o}, co-published with Springer Science+ Business Media BV~…}
}

@book{waltersIntroduction,
  title={An introduction to ergodic theory},
  author={Walters, P.},
  Series={Graduate Texts in Mathematics},
  volume={79},
  year={1982},
  publisher={Springer}
}

@article{BerZel-NiceRecurrence,
  title={Sets of large values of polynomial multi-correlation functions},
  author={Bergelson, V. and Zelada, R.},
  journal={arXiv:2605.23050},
  year={2026}
}

@article{zelada2023GaussianETDS,
    title={Mixing and rigidity along asymptotically linearly independent sequences},
  author={Zelada, R.},
  journal={Ergodic Theory and Dynamical Systems},
  volume={43},
  number={10},
  pages={3506--3537},
  year={2023},
}

@article{ZelIP0Khintchine2023,
  title={Failure of {K}hintchine-type results along the polynomial image of {I}{P}$_0$ sets},
  author={Zelada, R.},
   journal={Discrete and Continuous Dynamical Systems},
  volume={44},
  number={5},
  pages={1475--1494},
  year={2024},
}

@article {ZorinIPNilpotentSz,
    AUTHOR = {Zorin-Kranich, P.},
     TITLE = {A nilpotent {IP} polynomial multiple recurrence theorem},
   JOURNAL = {Journal d'Analyse Math{\'e}matique},
    VOLUME = {123},
      YEAR = {2014},
     PAGES = {183--225},
}
\bibliographystyle{amsplain}
\end{document}